\documentclass[reqno]{amsart}

\usepackage [mathscr]{eucal}
\usepackage{amsmath,amsthm,amssymb,amscd}
\usepackage{amsrefs}
\usepackage{epsf}
\usepackage{amsfonts}
\usepackage{dsfont}
\usepackage{latexsym}
\usepackage{mathrsfs}
\usepackage{layout}
\usepackage{bm}

\newtheorem{theorem}{Theorem}[section]
\newtheorem{lemma}[theorem]{Lemma}
\newtheorem{proposition}[theorem]{Proposition}
\newtheorem{corollary}[theorem]{Corollary}

\newtheorem{remark}[theorem]{Remark}
\newtheorem{definition}[theorem]{Definition}

\numberwithin{equation}{section}

\def\Xint#1{\mathchoice
{\XXint\displaystyle\textstyle{#1}}%
{\XXint\textstyle\scriptstyle{#1}}%
{\XXint\scriptstyle\scriptscriptstyle{#1}}%
{\XXint\scriptscriptstyle\scriptscriptstyle{#1}}%
\!\int}
\def\XXint#1#2#3{{\setbox0=\hbox{$#1{#2#3}{\int}$ }
\vcenter{\hbox{$#2#3$ }}\kern-.6\wd0}}

\def\dint{\Xint-}

\DeclareMathOperator *{\essosc}{ess\ osc}
\DeclareMathOperator *{\osc}{osc}
\DeclareMathOperator *{\esssup}{ess\ sup}
\DeclareMathOperator *{\essinf}{ess\ inf}
\DeclareMathOperator *{\di}{div} 
\DeclareMathOperator *{\meas}{meas}
\DeclareMathOperator *{\dist}{dist}
\DeclareMathOperator *{\data}{data}
\DeclareMathOperator *{\diam}{diam}
\DeclareMathOperator *{\loc}{loc}
\DeclareMathOperator *{\Lip}{Lip}
\DeclareMathOperator *{\Tr}{Tr}
\DeclareMathOperator *{\BMO}{BMO}
\DeclareMathOperator *{\Proj}{Proj}
\DeclareMathOperator *{\BBR}{\mathbb{R}}
\DeclareMathOperator *{\BBC}{\mathbb{C}}

\begin{document}
\title[Boundary value problems for elliptic systems]{The $L^p$ Dirichlet and Regularity problems for second order Elliptic Systems with application to the Lam\'e system.}

\author{Martin Dindo\v{s}}
\address{School of Mathematics, \\
         The University of Edinburgh and Maxwell Institute of Mathematical Sciences, UK}
\email{M.Dindos@ed.ac.uk}
\begin{abstract}
In the paper \cite{DHM} we have introduced new solvability methods for strongly elliptic second order systems in divergence form on a domains above a Lipschitz graph, satisfying $L^p$-boundary data for $p$ near $2$. 
The main novel aspect of our result is that it applies to operators with coefficients of limited regularity and applies to operators satisfying a natural Carleson condition that has been first considered in the scalar case. 

In this paper we extend this result in several directions. We improve the range of solvability of the $L^p$ Dirichlet problem to the interval $2-\varepsilon < p<\frac{2(n-1)}{(n-3)}+\varepsilon$, for systems in dimension $n=2,3$  in the range $2-\varepsilon < p<\infty$.  We do this by considering solvability of the Regularity problem (with boundary data having one derivative in $L^p$)  
 in the range $2-\varepsilon < p<2+\varepsilon$.

Secondly, we look at perturbation type-results where we can deduce solvability of the $L^p$ Dirichlet problem for one operator from known  $L^p$ Dirichlet solvability of  a \lq\lq close" operator (in the sense of Carleson measure). This leads to improvement of the main result of the paper \cite{DHM}; we establish solvability of the $L^p$ Dirichlet problem in the interval $2-\varepsilon < p<\frac{2(n-1)}{(n-2)}+\varepsilon$ under a much weaker (oscillation-type) Carleson condition. 

A particular example of the system where all these results apply is the Lam\'e operator  for isotropic inhomogeneous materials with Poisson ratio $\nu<0.396$. In this specific case further improvements of the solvability range are possible, see \cite{DLP}.
\end{abstract}

\maketitle

\section{Introduction}\label{S:Intro}

This paper is motivated by the known results concerning boundary value problems 
for second order elliptic equations in divergence form, when the coefficients satisfying 
a certain natural, minimal smoothness condition (refer \cite{DPP}, \cite{DPR}, \cite{KP}). It extends the results of the paper \cite{DHM} in several interesting and important directions. Because of this we maintain as closely as possible the notation introduced in  \cite{DHM}.

Let $\Omega\subset{\BBR}^{n}$ be a domain defined by a Lipschitz function $\phi$, that is
\begin{equation}\label{Eqqq-1}
\Omega=\{(x_0,x'):\,x_0>\phi(x')\}.
\end{equation}
Consider a second order elliptic operstor in divergence form given by
\begin{equation}\label{ES}
\mathcal{L}u=\left[ \partial_{i} \left(A_{ij}^{\alpha \beta}(x) \partial_{j} u_{\beta}\right)+B_i^{\alpha\beta}(x)\partial_iu_\beta\right]_{\alpha}
\end{equation}
for $i,j\in\{0,\ldots,n-1\}$ and $\alpha,\beta\in\{1,\ldots,N\}$ with real coefficients.
Here the solution $u:\Omega\to{\BBR}^N$ is a vector-valued function satisfying $\mathcal Lu=0$. 
When $N=1$ the equation is scalar and much more is known in this case (c.f. \cites{DPP,DPR} for real and \cites{DP, DP2, DP3} for the complex coefficients cases). We note that if the coefficients of \eqref{ES} are complex valued then such system can be rewritten as a real valued system with $2N$ equations (by writing separate equation for the real and imaginary parts of $u_\beta$) and hence it suffices to consider \eqref{ES} with real coefficients.

We also note that our assumption that the domain $\Omega$ has the form \eqref{Eqqq-1} (i.e. it is an unbounded set above a Lipschitz graph) is primarily for convenience, the arguments given in this paper can be modified to the case of a bounded Lipschitz domain instead or other cases (such as an infinite strip, etc). 

There are many differences between second order elliptic equations and elliptic systems. In general, there is no maximal principle for elliptic systems and the DeGiorgi - Nash - Moser theory that shows interior $C^\alpha$ regularity for scalar elliptic PDE might no longer hold.

This causes a number of new challenges to be dealt with; for example it forces us to work with a weaker version of the nontangential maximal function (defined using the $L^2$ averages). The lack of maximum principle removes the natural $L^\infty$ end-point for solvability of the $L^p$ Dirichlet problem and prevents us from interpolating between the $L^2$ and $L^\infty$ solvabilities. This means that $L^p$ solvability results for $p\ne 2$ have to be obtained using different methods.

We shall say that the bounded and measurable coefficients $A=[A^{\alpha \beta}_{ij}]$ are strongly elliptic (the condition \eqref{EllipA} is usually called the {\bf Legendre} condition)
if there exist constants $0<\lambda\leq\Lambda<\infty$ such that
\begin{equation}\label{EllipA}
\lambda|\eta|^{2}\leq\sum_{\alpha,\beta=1}^{N}\sum_{i,j=0}^{n-1} 
A_{ij}^{\alpha\beta}(x)\eta_{i}^{\alpha}\eta_{j}^{\beta}
\end{equation}
for all nonzero $\eta\in{\BBR}^{nN}$ and a.e. $x\in\Omega$. We shall denote by $\Lambda=\|A\|_{L^\infty(\Omega)}$. This is the strongest form of ellipticity and we shall assume it for most of our results. We note that most of our results can be stated under a weaker integral-type condition (such as (1.4) of \cite{DHM}) but in order to keep the arguments as simple as possible we choose not to do so.

For the results in the second half of this paper it will suffice to assume a weaker condition called {\bf Legendre-Hadamard} condition
\begin{equation}\label{EllipLH}
\lambda |p|^{2}|q|^2\leq\sum_{\alpha,\beta=1}^{N}\sum_{i,j=0}^{n-1} 
A_{ij}^{\alpha\beta}(x)p^\alpha p^\beta q_i q_j.
\end{equation}

It is easy to see that \eqref{EllipA} implies \eqref{EllipLH}, the converse is usually false except in certain special cases such as the scalar case $N=1$.

Recall that in the paper \cite{DHM} we have also assumed that the coefficients also satisfy a Carleson condition, namely that
\begin{equation}\label{Car_hatAA-x}
d{\mu}(x)=\left[\left(\sup_{B_{\delta(x)/2}(x)}|\nabla{A}|\right)^{2}+\left(\sup_{B_{\delta(x)/2}(x)}|B|\right)^{2}\right]
\delta(x)\,dx
\end{equation}
is a Carleson measure. It is of considerable interest to replace the condition \eqref{Car_hatAA-x} bya weaker Carleson condition, namely that
\begin{equation}\label{Car_hatAA-osc}
d{\mu}(x)=\left[\left(\osc_{B_{\delta(x)/2}(x)}{A}\right)^{2}\delta^{-1}(x)+\left(\sup_{B_{\delta(x)/2}(x)}|B|\right)^{2}\delta(x)\right]\,dx,
\end{equation}
is a Carleson measure. Here $\mbox{osc}_B A=\max_{i,j,\alpha,\beta}\left[\sup_B A_{ij}^{\alpha\beta}-\inf_B A_{ij}^{\alpha\beta}\right]$.
In the scalar case ($N=1$) it has been shown that \eqref{Car_hatAA-osc} is always sufficient for solvability of both Dirichlet and Regularity boundary value problems, provided the Carleson norm of \eqref{Car_hatAA-osc} is small.  This follows from solvability results for the Carleson condition \eqref{Car_hatAA-x} and Dahlberg-Kenig perturbation result for real and scalar elliptic PDEs. In the case of systems a similar perturbation result is not known.

We will establish an $L^2$ version of perturbation result for systems and use it to prove solvability of the $L^p$ Dirichlet boundary problem under the condition \eqref{Car_hatAA-osc} when $2-\varepsilon< p<\frac{2(n-1)}{n-2}+\varepsilon$.

Since we are not able to prove analogous perturbation result for the the $L^p$ Regularity problem, we only establish solvability of the Regularity problem for $2-\varepsilon< p<2+\varepsilon$ assuming the stronger Carleson condition on  \eqref{Car_hatAA-x}.

Based on this result we then improve solvability result for the $L^p$ Dirichlet boundary from \cite{DHM} to the interval $2-\varepsilon< p<\frac{2(n-1)}{(n-3)}+\varepsilon$ under the stronger Carleson condition for \eqref{Car_hatAA-x}.\vglue1mm

The literature on the solvability of boundary value problems for elliptic systems in domains of $\mathbb R^n$ is limited except 
when the tensor $A$ has constant coefficients, or at least smooth enough so that methods like boundary layer potentials may 
be employed. For the solvability the $L^p$-Dirichlet problem for constant coefficients second-order elliptic systems in the 
range $2-\varepsilon<p<2+\varepsilon$ see \cites{DKV, F, FKV, G, BM} and \cite{K}. It was subsequently shown in \cites{S1,S2} 
that in the constant coefficient case this range may be extended to the interval $2-\varepsilon<p<\frac{2(n-1)}{n-3}+\varepsilon$ 
by exploring the solvability of the Regularity problem. See also \cite{MMMM} and in particular \cite{S4} for more 
recent developments. We take advantage of \cite{S4} to extrapolate from solvability for $p=2$ to the range 
$2\le p < \frac{2(n-1)}{n-2}+\varepsilon$ without needing to establish the solvability of the Regularity problem.

Of notable interest is also paper  \cite{DM} where the Stationary Navier-Stokes system in nonsmooth manifolds
was studied. The authors have established results for $L^p$ solvability of the linearized Stokes operator with 
variable coefficients via the method of layer potentials. Because of the method used, at least H\"older continuity 
of the underlying metric tensor had to be assumed.

Another special case is when $A$ is of block-form. For operators $\mathcal{L}=\mbox{div}(A(x)\nabla\cdot)$ 
associated with block matrices $A$, there are numerous results on the $L^p$-solvability of the Dirichlet, Regularity, 
and Neumann problems. This body of results owes to the solution of the Kato problem, where the coefficients of the 
block matrix are also assumed to be independent of the transverse variable. This assumption is usually referred 
in literature as ``$t$-independent" (in our notation it is the $x_0$ variable). See \cite{AHLMT}, \cite{HM}, 
as well as a series of papers by Auscher, Rosen(Axelsson), and McIntosh for second-order elliptic systems (\cites{AA1, AR2, AAM}). 

There are also solvability results in various special cases, assuming that the solutions satisfy 
De Giorgi - Nash - Moser estimates; see \cite{AAAHK} and \cite{HKMPreg} for example (the latter paper is 
also concerned with operators that are $t$-independent). Finally, there are perturbation results in a multitude 
of special cases, such as \cite{AAH} and \cite{AAM}; the first  paper has $L^2$-solvability results for small $L^\infty$ perturbations of real elliptic operators when the complex matrix is $t$-independent and  the second paper shows that solvability in $L^2$ implies solvability 
in $L^p$ for $p$ near $2$.

Significantly, in the formulation of our solvability result for elliptic systems we shall not assume ``$t$-independence".
Instead, we assume the coefficients $A$ and $B$ satisfy a natural Carleson condition that has appeared 
in the literature so far only in the scalar case for real \cites{KP01,DPP,DPR} and complex valued \cites{DP,DP2,DP3} coefficients. 
The Carleson condition on $A$, formulated in \eqref{Car_hatAA} below, holds uniformly on Lipschitz sub-domains, 
and is therefore a natural condition in the context of chord-arc domains as well. However, in this work we do not  go beyond the class of Lipschitz domain.

We are ready to formulate our results. Our first theorem is a straightforward improvement of \cite[Theorem 1.1]{DHM}.
\begin{theorem}\label{S3:T0} 
Let $\Omega$ be the Lipschitz domain $\{(x_0,x')\in{\mathbb{R}}\times{\mathbb{R}}^{n-1}:\,x_0>\phi(x')\}$ 
with Lipschitz constant $L=\|\nabla\phi\|_{L^\infty}$. Assume that the coefficient tensor $A$ of the operator \eqref{ES} is strongly elliptic with constants $\lambda,\Lambda$ (c.f. \eqref{EllipA}). 
In addition assume that the following holds:
\begin{itemize}
\item[$(i)$] $A_{0j}^{\alpha\beta}=\delta_{\alpha\beta}\delta_{0j}$.
\item[$(ii)$]
\begin{equation}\label{Car_hatAA}
d{\mu}(x)=\left[\left(\osc_{B_{\delta(x)/2}(x)}{A}\right)^{2}\delta^{-1}(x)+\left(\sup_{B_{\delta(x)/2}(x)} |{B}|\right)^{2} \delta(x)\right]\,dx
\end{equation}
is a Carleson measure in $\Omega$. 
\end{itemize}

There exists a small $K=K(\lambda,\Lambda,n,N)>0$ such that if
\begin{equation}\label{Small-Cond}
\max\big\{L\,,\,\|\mu\|_{\mathcal C}\big\}\leq K
\end{equation}
then $L^p$-Dirichlet problem for the system \eqref{E:D} is solvable for all $2-\varepsilon< p<\frac{2(n-1)}{n-2}+\varepsilon$
and the estimate
\begin{equation}\label{Main-Est}
\|\tilde{N}_a u\|_{L^{p}(\partial \Omega)}\leq C\|f\|_{L^{p}(\partial \Omega;{\BBR}^N)}
\end{equation}
holds for all energy solutions $\mathcal Lu=0$ with datum $f$. Here $\varepsilon=\varepsilon(\lambda, \Lambda, n,N,K)>0$ and $C=C(\lambda, \Lambda, n,N,\Omega,K)>0$.
\end{theorem}

We have an improved solvability range under a stronger Carleson condition.

\begin{theorem}\label{S3:T1} 
Let $\Omega$ be the Lipschitz domain $\{(x_0,x')\in{\mathbb{R}}\times{\mathbb{R}}^{n-1}:\,x_0>\phi(x')\}$ 
with Lipschitz constant $L=\|\nabla\phi\|_{L^\infty}$. Assume that the coefficient tensor $A$ of the operator \eqref{ES} is strongly elliptic with constants $\lambda,\Lambda$ (c.f. \eqref{EllipA}). 
In addition assume that the following holds:
\begin{itemize}
\item[$(i)$] $A_{0j}^{\alpha\beta}=\delta_{\alpha\beta}\delta_{0j}$.
\item[$(ii)$]
\begin{equation}\label{Car_hatAAxx}
d{\mu}(x)=\left[\left(\sup_{B_{\delta(x)/2}(x)}|\nabla{A}|\right)^{2}
+\left(\sup_{B_{\delta(x)/2}(x)} |{B}|\right)^{2} \right]\delta(x)\,dx
\end{equation}
is a Carleson measure in $\Omega$. 
\end{itemize}

There exists a small $K=K(\lambda,\Lambda,n,N)>0$ such that if
\begin{equation}\label{Small-Condx}
\max\big\{L\,,\,\|\mu\|_{\mathcal C}\big\}\leq K
\end{equation}
then $L^p$-Dirichlet problem for the system \eqref{E:D} is solvable for all $2-\varepsilon< p<\frac{2(n-1)}{n-3}+\varepsilon$
and the estimate
\begin{equation}\label{Main-Estx}
\|\tilde{N}_a u\|_{L^{p}(\partial \Omega)}\leq C\|f\|_{L^{p}(\partial \Omega;{\BBR}^N)}
\end{equation}
holds for all energy solutions $\mathcal Lu=0$ with datum $f$. Here $\varepsilon=\varepsilon(\lambda, \Lambda, n,N,K)>0$ and $C=C(\lambda, \Lambda, n,N,\Omega,K)>0$.
\end{theorem}

Finally, we have an improved regularity of solutions (so-called Regularity problem):

\begin{theorem}\label{S3:T3} 
Let $\Omega$ be the Lipschitz domain $\{(x_0,x')\in{\mathbb{R}}\times{\mathbb{R}}^{n-1}:\,x_0>\phi(x')\}$ 
with Lipschitz constant $L=\|\nabla\phi\|_{L^\infty}$. Consider the operator
\begin{equation}\label{ES-4}
\widetilde{\mathcal{L}}u=\left[ \partial_{i} \left(\widetilde{A_{ij}^{\alpha \beta}}(x) \partial_{j} u_{\beta}\right)\right]_{\alpha}
\end{equation}
and assume that the operator \eqref{ES-4} can be rewritten as a system with first-order terms
\begin{equation}\label{ES-2}
\mathcal{L}u=\left[ \partial_{i} \left(A_{ij}^{\alpha \beta}(x) \partial_{j} u_{\beta}\right)+B_i^{\alpha\beta}(x)\partial_iu_\beta\right]_{\alpha}
\end{equation}
whose coefficients satisfy
\begin{itemize}
\item[$(i)$] $A_{0j}^{\alpha\beta}=\delta_{\alpha\beta}\delta_{0j}$.
\item[$(ii)$]
\begin{equation}\label{Car_hatAAxxx}
d{\mu}(x)=\left[\left(\sup_{B_{\delta(x)/2}(x)}|\nabla{A}|\right)^{2}
+\left(\sup_{B_{\delta(x)/2}(x)} |{B}|\right)^{2} \right]\delta(x)\,dx
\end{equation}
is a Carleson measure in $\Omega$. 
\item[$(iii)$]The coefficient tensor $A$ is strongly elliptic. 
\end{itemize}
There exists a small $K=K(\lambda,\Lambda,n,N)>0$ such that if
\begin{equation}\label{Small-Condxx}
\max\big\{L\,,\,\|\mu\|_{\mathcal C}\big\}\leq K
\end{equation}
then $L^p$-Regularity problem for the system ${\mathcal L}u=\widetilde{{\mathcal L}}u=0$  is solvable for all $2-\varepsilon< p<2+\varepsilon$
and the estimate
\begin{equation}\label{Main-Estxx}
\|\tilde{N}_a (\nabla u)\|_{L^{p}(\partial \Omega)}\leq C\|\nabla_T f\|_{L^{p}(\partial \Omega;{\BBR}^N)}
\end{equation}
holds for all energy solutions $\mathcal Lu=0$ with datum $f$. Here $\varepsilon=\varepsilon(\lambda, \Lambda, n,N,K)>0$ and $C=C(\lambda, \Lambda, n,N,\Omega,K)>0$.
\end{theorem}

\begin{remark} In regards the structural condition $(i)$ in all three previous theorems section 2 of paper \cite{DHM} outlines
how any PDE of the form \eqref{ES} or \eqref{ES-4} can be rewritten so that the condition $(i)$ holds. See there also for discussion on how such coefficients change affects strong ellipticity and the Carleson condition.
\end{remark}

\noindent{\bf Example.} Consider the Lam\'e operator ${\mathcal L}$ for isotropic inhomogeneous materials
in a domain $\Omega$ with Lam\'e coefficients $\lambda(x)$ and $\mu(x)$. Then for $u:\Omega\to{\mathbb R}^n$ in vector notation (c.f. \cite{UW}) $\mathcal L$ has the form
\begin{equation}\label{Lame}
\mathcal Lu=\nabla\cdot\left(\lambda(x)(\nabla\cdot u)I+\mu(x)(\nabla u+(\nabla u)^T) \right).
\end{equation}
We can write the operator ${\mathcal L}$ in the form \eqref{ES} where
\begin{equation}\label{eqLM}
{A}_{ij}^{\alpha\beta}(x)=\mu(x)\delta_{ij}\delta_{\alpha\beta}+\lambda(x)\delta_{i\alpha}\delta_{j\beta}+\mu(x)\delta_{i\beta}\delta_{j\alpha}
\end{equation}
and ${B}^{\alpha\beta}_i(x)=0$.

We show that we can apply our results to the Lam\'e operator for isotropic inhomogeneous materials, provided we also have strong ellipticity for the modified operator that also satisfies $(i)$ and $(ii)$. The paper \cite{DHM} discusses this detail in section 7. As shown there (Lemma 7.1) the Lam\'e system can be rewritten so that $(i)$ holds and the operator is strongly elliptic if 
$$\lambda<(\sqrt{8}+1)\mu\approx 3.828\mu,$$
or alternatively the Poisson ratio $\nu:=\frac{\lambda}{2(\lambda+\mu)}<0.396$. There are many materials where this holds (for example  aluminium, bronze, steel and many other metals, carbon,  polystyrene, PVC, silicate glasses, concrete, etc) \cite{MR}. Examples of few materials where this assumption fails are gold, lead or rubber. For these three materials $\nu$ is near the incompressibility limit ($\nu=\frac12-$) at which  \eqref{Lame} gives {\rm div}$\,u=0$, i.e., the material is incompressible. Intuitively, as both gold and lead are very soft metals, under pressure they behave as liquids, that is a pressure in one direction will cause them to change shape and stretch in remaining directions in order to preserve volume. Rubber is nearly incompressible with $\nu\approx 0.49$.
\vglue1mm

Hence we obtain the following corollary:
\vglue1mm

\begin{corollary} \label{C:lame}
Let $\Omega$ be the Lipschitz domain $\{(x_0,x')\in{\mathbb{R}}\times{\mathbb{R}}^{n-1}:\,x_0>\phi(x')\}$ 
with Lipschitz constant $L=\|\nabla\phi\|_{L^\infty}$. Assume that the Lame coefficients $\lambda,\mu\in L^\infty(\Omega)$ satisfy the following:
\begin{itemize} 
\item[$(i)$] There exists $\mu_0>0$ such that
\begin{equation}\label{Cond-lame}
\mbox{\rm ess }\inf_{x\in\Omega}\{(\sqrt{8}-1)\mu(x)+\lambda(x),(\sqrt{8}+1)\mu(x)-\lambda(x)\}\ge \mu_0.
\end{equation}\item[$(ii)$]
\begin{equation}\label{Car_lame2}
d{\nu}(x)=\left[\left(\osc_{B_{\delta(x)/2}(x)}{\lambda}\right)^{2}+\left(\osc_{B_{\delta(x)/2}(x)}{\mu}\right)^{2}\right]\delta^{-1}(x)
\end{equation}
is a Carleson measure in $\Omega$. 
\end{itemize}
There exists a small $K=K(\lambda,\Lambda,n,N)>0$ such that if
\begin{equation}\label{Small-Cond2}
\max\big\{L\,,\,\|\nu\|_{\mathcal C}\big\}\leq K
\end{equation}
then $L^p$-Dirichlet problem for the Lam\'e system 
\begin{equation}\label{ES-lame2}
\begin{cases}
\mathcal Lu=\nabla\cdot\left(\lambda(x)(\nabla\cdot u)I+\mu(x)(\nabla u+(\nabla u)^T) \right)=0 
& \text{in } \Omega,
\\[4pt]
u(x)=f(x) & \text{ for $\sigma$-a.e. }\,x\in\partial\Omega, 
\\[4pt]
\tilde{N}_a(u) \in L^{p}(\partial \Omega), &
\end{cases}
\end{equation}
is solvable for all $2-\varepsilon< p<\frac{2(n-1)}{n-2}+\varepsilon$ and the estimate
\begin{equation}\label{Main-Est-LM}
\|\tilde{N}_a u\|_{L^{p}(\partial \Omega)}\leq C\|f\|_{L^{p}(\partial \Omega;{\BBR}^n)}
\end{equation}
holds for all energy solutions $u:\Omega\to {\mathbb R}^n$ with datum $f$. Here\newline $\varepsilon=\varepsilon(\mu_0,\|\lambda\|_{L^\infty},\|\mu\|_{L^\infty},n)>0$ and $C=C(\mu_0,\|\lambda\|_{L^\infty},\|\mu\|_{L^\infty},n)>0$. 

The solvability range 
improves to $2-\varepsilon< p<\frac{2(n-1)}{n-3}+\varepsilon$ if
\begin{equation}\label{Car_lame}
d{\nu}(x)=\sup_{B_{\delta(x)/2}(x)}\left(|\nabla{\lambda}|+|\nabla{\mu}|\right)^{2}\delta(x)\,dx
\end{equation}
is a Carleson measure with small Carleson norm.

Finally, in the range $2-\varepsilon< p<2+\varepsilon$ we also have the estimate
\begin{equation}\label{Main-Est-LMR}
\|\tilde{N}_a (\nabla u)\|_{L^{p}(\partial \Omega)}\leq C\|\nabla_T f\|_{L^{p}(\partial \Omega;{\BBR}^n)}
\end{equation}
if \eqref{Car_lame} is a Carleson measure with small Carleson norm.
\end{corollary}

\noindent{\it Remark.} In particular, the part of Corollary \ref{C:lame}  where the condition \eqref{Car_lame2} is assumed does apply to objects composed of multiple materials joined together by reasonably smooth interfaces between them (say Lipschitz) that might extend all the way to the boundary. The Carleson measure $\nu$ then measures jumps in coefficients $\lambda$ and $\mu$  across the different materials.\vglue1mm

\noindent{\it Remark 2.} Observe that under condition that \eqref{Car_lame2} is a small Carleson measure we obtain solvability in the range $p\in(2-\varepsilon,\frac{2(n-1)}{n-2}+\varepsilon)$ using the extrapolation. Recall that the idea of extrapolation is to deduce from solvability at a single point, say $L^2$, solvability in certain range of values of $p\in (2-\varepsilon,q(n))$. In the results above we have used a recent result of Shen \cite{S4} which shows $q(n)>\frac{2(n-1)}{n-2}$ for elliptic systems. Note that when $N=1$ then $q(n)=\infty$ by the maximum principle but when $N>1$ the question of the optimal value of $q(n)$ is open. In our upcoming work with J. Li and J. Pipher \cite{DLP} we improve Shen's result for the Lam\'e systems satisfying the ellipticity condition \eqref{C:lame} and show that
$$q(2)=\infty,\qquad q(3)>11.50,\quad q(4)>8.055,\qquad q(n)>\frac{2}{1-\sqrt{8\sqrt{2}-11}}\approx 4.546.$$

Observe that this is improves the solvability range claimed above even under the stronger Carleson condition \eqref{Car_lame} in dimensions $n>3$ where the range $p\in(2-\varepsilon, \frac{2(n-1)}{n-3}+\varepsilon)$ is stated.

Our improvement is based on the concept of $p$-ellipticity which in \cite{DLP} we introduce for elliptic systems. This concept  in an easier setting of scalar complex valued PDEs and its connection to extrapolation appears in our most recent work \cite{DP4}.

\vglue2mm

The paper is organised as follow. In Section~\ref{S2}, we introduce important notions and definitions needed later. Section 3 precisely defines the notion of solvability of Dirichlet and Regularity boundary value problems. 
Section 4 establishes bounds for the square function by the boundary data and the nontangential maximal function.
Reverse estimates are in section 5. Then, in section 6 we prove Theorem \ref{S3:T0},  Theorem \ref{S3:T3} in section 7 and Theorem \ref{S3:T1}  in section 8. Finally in section 9 we establish Corollary \ref{C:lame} for the Lam\'e system, which requires somewhat delicate application of Theorems \ref{S3:T0}-\ref{S3:T3}.

\section{Definitions and background results}
\label{S2}

For a vector-valued function $u=(u_{\alpha})_{1\leq\alpha\leq N}:\Omega\to{\BBR}^{N}$ we let 
$\nabla u$ denote the Jacobian matrix of $u$. The latter is defined as the matrix with entries
\begin{equation}\label{Eqqq-2} 
\left(\nabla u\right)_{i}^{\alpha}=\partial_{i} u_{\alpha} 
=\frac{\partial u_{\alpha}}{\partial x_{i}}
\end{equation}
for $i\in\{0,\ldots,n-1\}$ and $\alpha\in\{1,\ldots,N\}$.

Given an open set $\Omega\subseteq{\mathbb{R}}^n$, for $0\leq k\leq\infty$ we use $C^{k}(\Omega;{\BBR}^{N})$ 
to denote the space of all ${\BBR}^{N}$-valued functions in $\Omega$ with continuous partial derivatives up to 
order $k$. Also, we shall let $C^{k}_{0}(\Omega;{\BBR}^{N})$ be the subspace of $C^{k}(\Omega;{\BBR}^{N})$ 
consisting functions that are compactly supported in $\Omega$. For $k\in{\mathbb{N}}$ and $1\leq p<\infty$, 
let $W^{k,p}(\Omega;{\BBR}^{N})$ be the Sobolev space which is the collection of ${\BBR}^{N}$-valued locally 
integrable functions in $\Omega$ having distributional derivatives of order $\leq k$ in $L^p(\Omega;{\BBR}^{N})$. 
When $k=1$, equip this space with the norm
\begin{equation}\label{EFFV} 
\|u\|_{W^{1,p}(\Omega)}:=\left[\int_{\Omega}\left(|u(x)|^{p}+|(\nabla u)(x)|^{p}\right)\,dx\right]^{1/p}.
\end{equation}
Also, let $W^{k,p}_{\rm loc}(\Omega;{\BBR}^{N})$ stands for the local version of $W^{k,p}(\Omega;{\BBR}^{N})$. 
Similarly, we denote by $\dot{W}^{k,p}(\Omega;{\BBR}^{N})$ the homogeneous version of the $L^p$-based Sobolev 
space of order one in $\Omega$. When $k=1$, this is endowed with the semi-norm
\begin{equation}\label{EFFV2} 
\|u\|_{\dot{W}^{1,p}(\Omega)}:=\left[\int_\Omega|(\nabla u)(x)|^{p}\,dx\right]^{1/p}.
\end{equation}

Throughout this paper, by a weak solution of \eqref{ES} in $\Omega$ we shall understand a function 
$u\in W^{1,2}_{\rm loc}(\Omega;{\BBR}^{N})$ satisfying $\mathcal{L}u=0$ in the sense of distributions 
in $\Omega$. 

\subsection{Non-tangential maximal and square functions}
\label{SS:NTS}

Consider a domain of the form 
\begin{equation}\label{Omega-111}
\Omega=\{(x_0,x')\in\BBR\times{\BBR}^{n-1}:\, x_0>\phi(x')\},
\end{equation}
where $\phi:\BBR^{n-1}\to\BBR$ is a Lipschitz function with Lipschitz constant given by 
$L:=\|\nabla\phi\|_{L^\infty(\BBR^{n-1})}$. For each point $x\in{\mathbb{R}}^n$ abbreviate 
$\delta(x):=\mbox{dist}(x,\partial\Omega)$. In particular, %
\begin{equation}\label{PTFCC}
\delta(x)\approx x_0-\phi(x')\,\,\text{ uniformly for }\,\,x=(x_0,x')\in\Omega.
\end{equation}

A cone {\rm (}or non-tangential approach region{\rm )} of aperture $a\in(0,\infty)$ 
with vertex at the point $Q=(x_0,x')\in{\mathbb{R}}\times{\mathbb{R}}^{n-1}$ is defined as
\begin{equation}\label{TFC-6}
\Gamma_{a}(Q)=\big\{y=(y_0,y')\in{\mathbb{R}}\times{\mathbb{R}}^{n-1}:\,a(y_0-x_0)>|x'-y'|\big\}.
\end{equation}
Imposing the demand that $a\in(0,1/L)$ then ensures that $\Gamma_{a}(Q)\subseteq\Omega$ whenever $Q\in\partial\Omega$.
In particular, when $\Omega=\BBR^n_+$ all parameters $a\in(0,\infty)$ may be considered.
Sometimes it is necessary to truncate $\Gamma_{a}(Q)$ at height $h$, in which scenario we write
\begin{equation}\label{TRe3}
\Gamma_{a}^{h}(Q):=\Gamma_{a}(Q)\cap\{x\in\Omega:\,\delta(x)\leq h\}.
\end{equation}

\begin{definition}\label{D:S}
For $\Omega \subset \mathbb{R}^{n}$ as above and $a\in(0,1/L)$, the square function of some 
$u\in W^{1,2}_{\rm loc}(\Omega; {\BBR}^{N})$ is defined at each $Q\in\partial\Omega$ by
\begin{equation}\label{yrdd}
S_{a}(u)(Q):=\left(\int_{\Gamma_{a}(Q)}|(\nabla u)(x)|^{2}\delta(x)^{2-n}\,dx\right)^{1/2}
\end{equation}
and, for each $h>0$, its truncated version is given by 
\begin{equation}\label{yrdd.2}
S_{a}^{h}(u)(Q):=\left(\int_{\Gamma_{a}^{h}(Q)}|(\nabla u)(x)|^{2}\delta(x)^{2-n}\,dx\right)^{1/2}.
\end{equation}
\end{definition}

A simple application of Fubini's theorem gives 
\begin{equation}\label{SSS-1}
\|S_{a}(u)\|^{2}_{L^{2}(\partial\Omega)}\approx\int_{\Omega}|(\nabla u)(x)|^{2}\delta(x)\,dx.
\end{equation}

\begin{definition}\label{D:NT.a} 
For $\Omega\subset\mathbb{R}^{n}$ as above and $a\in(0,1/L)$, the nontangential maximal function of some 
$u\in C^{\,0}(\Omega;{\BBR}^{N})$ and its truncated version at height $h$ are defined at each $Q\in\partial\Omega$ by
\begin{equation}\label{SSS-2}
N_{a}(u)(Q):=\sup_{x\in\Gamma_{a}(Q)}|u(x)|\,\,\text{ and }\,\,
N^h_{a}(u)(Q):=\sup_{x\in\Gamma^h_{a}(Q)}|u(x)|.
\end{equation}
\end{definition}

Moreover, we shall also consider a related version  of the above nontangential maximal function.
This is denoted by $\tilde{N}_a$ and is defined using $L^2$ averages over balls in the domain $\Omega$. 
Specifically, we make the following definition.

\begin{definition}\label{D:NT.b} 
For $\Omega\subset\mathbb{R}^{n}$ as above and $a\in(0,1/L)$,
given $u\in L^2_{\rm loc}(\Omega;{\BBR}^{N})$ we set
\begin{equation}\label{SSS-3}
\tilde{N}_{a}(u)(Q):=\sup_{x\in\Gamma_{a}(Q)}w(x)\,\,\text{ and }\,\,
\tilde{N}_{a}^{h}(u)(Q):=\sup_{x\in\Gamma_{a}^{h}(Q)}w(x)
\end{equation}
for each $Q\in\partial\Omega$ and $h>0$ where, at each $x\in\Omega$, 
\begin{equation}\label{w}
w(x):=\left(\dint_{B_{\delta(x)/2}(x)}|u|^{2}(z)\,dz\right)^{1/2}.
\end{equation}
\end{definition}

Here and elsewhere, a barred integral indicates integral average.
We note that, given $u\in L^2_{\rm loc}(\Omega;{\BBR}^{N})$, the function $w$ 
associated with $u$ as in \eqref{w} is continuous and $\tilde{N}_a(u)=N_a(w)$ 
everywhere on $\partial\Omega$. For systems with bounded measurable coefficients, the best regularity we can expect 
from a weak solution of \eqref{ES} is $u\in W^{1,2}_{\rm loc}(\Omega;\BBR^N)$. 
In particular, $u$ might not be pointwise well-defined. In the scalar case $N=1$ 
by the De Giorgi-Nash-Moser estimates the situation is different as the solutions are locally H\"older continuous. 
Hence, while in the scalar case considering $N_a$ typically suffices, in the case of systems the consideration of 
$\tilde{N}_a$ becomes necessary. In particular, this is the case when we assume the Carleson condition on \eqref{Car_hatAA}.

\subsection{The Carleson measure condition}
\label{SS:Car} 

We begin by recalling the definition of a Carleson measure in a domain $\Omega$ as in \eqref{Omega-111}. 
For $P\in{\BBR}^n$, define the ball centered at $P$ with the radius $r>0$ as
\begin{equation}\label{Ball-1}
B_{r}(P):=\{x\in{\BBR}^n:\,|x-P|<r\}.
\end{equation}
Next, given an arbitrary location $Q \in \partial\Omega$ along with a scale $r>0$, we shall abbreviate 
$\Delta=\Delta_{r}=\Delta_{r}(Q)=\partial\Omega\cap B_{r}(Q)$ and refer to this as the 
surface ball centered at $Q$ and of radius $r$. The Carleson region $T(\Delta_r(Q))$ associated with such a surface ball 
is then defined by
\begin{equation}\label{tent-1}
T(\Delta_{r}(Q)):=\Omega\cap B_{r}(Q).
\end{equation}

\begin{definition}\label{Carleson}
Let $\Omega$ be as in \eqref{Omega-111}.
A Borel measure $\mu$ in $\Omega$ is said to be Carleson if it has finite Carleson norm, i.e., 
\begin{equation}\label{CMC-1}
\|\mu\|_{\mathcal C}=\sup_{\Delta}\frac{\mu\left(T(\Delta)\right)}{\sigma(\Delta)}<\infty,
\end{equation}
where the supremum is taken over all surface balls $\Delta\subseteq\partial\Omega$, and 
where $\sigma$ is the surface measure on $\partial\Omega$. 
\end{definition}

The following result plays a significant role in the future (c.f. \cite{DHM}).

\begin{proposition}\label{T:Car}
Let $\Omega$ be as in \eqref{Omega-111} and fix some $a\in(0,1/L)$. Given a function $f\in L^\infty_{\rm loc}(\Omega)$, define 
$d\nu=f\,dx$ and $d\mu(x)=\left[\esssup_{B_{\delta(x)/2}(x) }|f|\right]dx$. Assume that $\mu$ is a Carleson measure in $\Omega$. 
Then there exists a finite constant $C=C(L,a)>0$ such that for every $u\in L^{2}_{\rm loc}(\Omega;{\BBC})$ one has
\begin{equation}\label{Ca-222}
\int_{\Omega}|u(x)|^2\,d\nu(x)\leq C\|\mu\|_{\mathcal{C}} 
\int_{\partial\Omega}\left(\tilde{N}_{a}(u)\right)^2\,d\sigma.
\end{equation}
\end{proposition}

Moreover, the aforementioned assumption on coefficients of the system \eqref{ES} 
is compatible with the useful change of variables described in the next two subsections. 

\subsection{Inequalities}
\label{SS:Ineq}

Here we recall some of basic the inequalities that hold for weak solutions of the operator 
${\mathcal{L}}$. 

\begin{proposition}\label{poincare}{\rm (}Poincar\'e inequality{\rm )}
There exists a finite dimensional constant $C=C(n)>0$ such that, for all balls $B_{R}\subset{\BBR}^{n}$ 
and all $u\in W^{1,2}(B_{R};{\BBR}^{N})$, 
\begin{gather*}
\int_{B_{R}}|u-u_{B_{R}}|^{2}\,dx\leq CR^{2}\int_{B_{R}}|\nabla u|^{2}\,dx, 
\end{gather*}
where
\begin{equation}\label{uuu-AVE}
u_{B_{R}}:=\dint_{B_{R}}u(x)\,dx.
\end{equation}
\end{proposition}

\begin{proposition}\label{caccio}{\rm (}Cacciopoli inequalities{\rm )} 
Let $\Omega$ be as in \eqref{Eqqq-1}, and let $\mathcal L$ as in \eqref{ES} satisfy 
the Legendre condition \eqref{EllipA}. In addition, assume that there exists some $M\in(0,\infty)$ 
with the property that for a.e. $x\in\Omega$ one has $|B(x)|\le M\delta^{-1}(x)$.

Then there exists a finite positive constant $C=C(n,N,\lambda,\Lambda, M)>0$ such that if $B_{2R}\subset\Omega$
and $u\in W^{1,2}(B_{2R};{\BBR}^{N})$ solves $\mathcal Lu=0$ in $B_{2R}$ it follows that 
\begin{gather*}
\int_{B_{R}}|\nabla u|^{2}\,dx\leq CR^{-2}\int_{B_{2R}}|u|^{2}\,dx. 
\end{gather*}

Additionally, if $M>0$ is sufficiently small, there also exists a finite positive constant $C=C(n,N,\lambda,\Lambda)>0$ such that if $R>0$ and 
$u\in W^{1,2}(T(\Delta_{2R}))$ satisfies $\mathcal Lu=0$ in $T(\Delta_{2R})$ as well as 
${\rm Tr}\,u=0$ on $\Delta_{2R}$ then
\begin{gather*}
\int_{T(\Delta_{R})}|\nabla u|^{2}\,dx\leq CR^{-2}\int_{T(\Delta_{2R})}|u|^{2}\,dx. 
\end{gather*}
\end{proposition}

\begin{proposition}\label{Ciac2}(A Cacciopoli inequality for the second gradient) Let $\mathcal L$ as in \eqref{ES} satisfy 
the Legendre-Hadamard condition \eqref{EllipLH}. In addition assume that $|\nabla A|,|\boldsymbol{B}|\le M/R$ on $B_{2R}$ for some $M>0$. 

Then there exists a finite positive constant $C=C(n,N,\lambda,\Lambda, M)>0$ such that if $B_{2R}\subset\Omega$
and $u\in W^{1,2}(B_{2R};{\BBR}^{N})$ solves $\mathcal Lu=0$ in $B_{2R}$ it follows that
\begin{gather*}
\int_{B_{R}}|\nabla^2 u|^{2}\,dx\leq CR^{-2}\int_{B_{2R}}|\nabla u|^{2}\,dx. 
\end{gather*}
\end{proposition}

\begin{proof} We will be brief and only outline the main idea. For each $0\le i\le n-1$ let $v_i=\partial_{x_i}u$. We write PDE system each function $v_i:B_{2R}\to \mathbb R^N$ satisfy. As $A$ is differentiable we can easily see that for each $i$ we have
$${\mathcal L}_i v_i=0,\qquad\mbox{where}\qquad {\mathcal L}_i={\mathcal L} + \mbox{1st order terms},$$
where ${\mathcal L}$ is our original operator (i.e., ${\mathcal L}u=0$). Now we follow the argument as in the proof of Proposition 2.7 of \cite{DHM} for each $v_i$. The only difference is that we have one extra first order term (when the derivative $\partial_{x_i}$ falls on $B$). That term does not cause an issue (we move the derivative $\partial_{x_i}$ from $B$ via integration by parts). We ultimately obtain by the same argument as in \cite{DHM} an estimate
\begin{gather*}
\int_{B_{R}}|\nabla v_i|^{2}\,dx\leq CR^{-2}\int_{B_{2R}}|\nabla u|^{2}\,dx. 
\end{gather*}
for a each $i$. From this our claim follows.
\end{proof}

\section{The $L^p$-Dirichlet and Regularity problems}
\label{S3}

We are ready to define the $L^p$-Dirichlet problem essentially following \cite{DHM}.
We first recall the classical solvability via the Lax-Milgram lemma in a domain 
$\Omega$ as in \eqref{Omega-111}. Recall, that under assumptions of strong ellipticity 
it can be shown via standard arguments that given any $f\in \dot{B}^{2,2}_{1/2}(\partial\Omega;{\BBR}^N)$ 
(this is the space of traces of functions in $\dot{W}^{1,2}(\Omega;{\BBR}^N)$) there exists a unique 
$u\in \dot{W}^{1,2}(\Omega;{\BBR}^N)$ such that $\mathcal{L}u=0$ in $\Omega$ for $\mathcal L$ 
given by \eqref{ES} and ${\rm Tr}\,u=f$ on $\partial\Omega$. 
We will call such $u\in \dot{W}^{1,2}(\Omega;{\BBR}^N)$ the {\it energy solution} of the elliptic system 
$\mathcal{L}$ in $\Omega$. For details see the beginning of section 3 of \cite{DHM} where the bilinear form is properly defined and its properties required for the application of the Lax-Milgram lemma such as its boundedness and coercivity are verified.

With this in hand, we can now define the notion of $L^p$ solvability. 

\begin{definition}\label{D:Dirichlet} 
Let $\Omega$ be the Lipschitz domain introduced in \eqref{Omega-111} and fix an integrability exponent 
$p\in(1,\infty)$. Also, fix a background parameter $a>0$. Consider the following Dirichlet problem 
for a vector valued function $u:\Omega\to{\BBR}^N$:
\begin{equation}\label{E:D}
\left\{
\begin{array}{l}
0=\partial_{i}\left(A_{ij}^{\alpha \beta}(x)\partial_{j}u_{\beta}\right) 
+B_{i}^{\alpha\beta}(x)\partial_{i}u_{\beta}\,\text{ in }\,\Omega,\,\,\,\alpha\in\{1,2,\dots,N\},
\\[6pt]
u(x)=f(x)\,\,\,\text{ for $\sigma$-a.e. }\,\,x\in\partial\Omega, 
\\[6pt]
\tilde{N}_a(u) \in L^{p}(\partial \Omega), 
\end{array}
\right.
\end{equation}
where the usual summation convention over repeated indices {\rm (}$i,j$ and $\beta$ in this case{\rm )} 
is employed. We say the Dirichlet problem \eqref{E:D} is solvable for a given $p\in(1,\infty)$ if there exists a finite constant 
$C=C(\lambda, \Lambda, n, p,\Omega)>0$ such that the unique energy solution $u\in \dot{W}^{1,2}(\Omega;{\BBR}^N)$, 
provided by the Lax-Milgram lemma, corresponding to a boundary datum 
$f\in L^p(\partial\Omega;{\BBR}^N)\cap \dot{B}^{2,2}_{1/2}(\partial\Omega;{\BBR}^N)$, satisfies the estimate
\begin{equation}\label{y7tGV}
\|\tilde{N}_a u\|_{L^{p}(\partial\Omega)}\leq C\|f\|_{L^{p}(\partial\Omega;{\BBR}^N)}.
\end{equation}
In \eqref{E:D} the solution $u$ agrees with $f$ at the boundary in the sense of trace on 
$\dot{W}^{1,2}(\Omega;{\BBR}^N)$ as well as in the sense of a.e. limit \eqref{eq-ae}, as explained below.
\end{definition}

\noindent{\it Remark.} By Lax-Milgram lemma the solution $u$ of \eqref{E:D} is unique 
in the space $\dot{W}^{1,2}(\Omega;{\BBR}^N)$ modulo constants (in ${\BBR}^N$). 
Our additional assumption that at $\sigma$-a.e. point on $\partial\Omega$ we have $u=f\in L^p(\partial\Omega;{\BBR}^N)$  
eliminates the constant solutions and, hence, guarantees genuine uniqueness. Since the space 
$\dot{B}^{2,2}_{1/2}(\partial\Omega;{\BBR}^N)\cap L^p(\partial\Omega;{\BBR}^N)$ is dense in 
$L^p(\partial\Omega;{\BBR}^N)$ for each $p\in(1,\infty)$, it follows that there exists a 
unique continuous extension of the solution operator
\begin{equation}\label{Sol-OP}
f\mapsto u
\end{equation}
to the whole space $L^p(\partial\Omega;{\BBR}^N)$, with $u$ such that $\tilde{N}_a u\in L^p(\partial\Omega)$ 
and the accompanying estimate $\|\tilde{N}_a u \|_{L^{p}(\partial \Omega)} 
\leq C\|f\|_{L^{p}(\partial\Omega;{\BBR}^N)}$ being valid. It is a legitimate question to consider in what sense
we have a convergence of $u$ given by the solution operator \eqref{Sol-OP} to its boundary datum 
$f\in L^{p}(\partial\Omega;{\BBR}^N)$. The answer can be found in the appendix of paper \cite{DP} 
(the proof is given for scalar operators but adapts in a straightforward way to our situation). 
Consider the average $u_{av}:\Omega\to {\mathbb R}^N$ defined by
$$
{u}_{av}(x)=\dint_{B_{\delta(x)/2}(x)} u(y)\,dy,\quad \forall x\in \Omega.
$$
Then 
\begin{equation}\label{eq-ae}
f(Q)=\lim_{x\to Q,\,x\in\Gamma(Q)} u_{av}(x),\qquad\text{for $\sigma$-a.e. }Q\in\partial\Omega.
\end{equation}

We can similarly define the Regularity problem.
\begin{definition}\label{D:Regularity} 
Let $\Omega$ be the Lipschitz domain introduced in \eqref{Omega-111} and fix an integrability exponent 
$p\in(1,\infty)$. Also, fix a background parameter $a>0$. Consider the following Dirichlet problem 
for a vector valued function $u:\Omega\to{\BBR}^N$:
\begin{equation}\label{E:R}
\begin{cases}
0=\partial_{i}\left(A_{ij}^{\alpha \beta}(x)\partial_{j}u_{\beta}\right) 
+B_{i}^{\alpha\beta}(x)\partial_{i}u_{\beta} 
& \text{in } \Omega,\quad\alpha\in\{1,2,\dots,N\}
\\[4pt]
u(x)=f(x) & \text{ for $\sigma$-a.e. }\,x\in\partial\Omega, 
\\[4pt]
\tilde{N}_a(\nabla u) \in L^{p}(\partial \Omega), &
\end{cases}
\end{equation}
where the usual Einstein summation convention over repeated indices ($i,j$ and $\beta$ in this case) 
is employed. We say the Regularity problem \eqref{E:R} is solvable for a given $p\in(1,\infty)$ if there exists 
$C=C(\lambda, \Lambda, n, p,\Omega)>0$ such that any energy solution $u\in \dot{W}^{1,2}(\Omega;{\BBR}^N)$, 
provided by the Lax-Milgram lemma, corresponding to a boundary datum 
$f$ with $\nabla_Tf \in L^p(\partial\Omega;{\BBR}^N)$ satisfies the estimate
\begin{equation}\label{y7tGVd}
\|\tilde{N}_a (\nabla u)\|_{L^{p}(\partial\Omega)}\leq C\|\nabla_T f\|_{L^{p}(\partial\Omega;{\BBR}^N)}.
\end{equation}
Here $\nabla_T$ denotes the ${\mathcal H}^{n-1}$-a.e. defined tangential gradient on $\partial\Omega$.
\end{definition}

\section{Estimates for the square function $S(u)$ and $S(\nabla u)$ of a solution}
\label{S4}

In this section we establish a one sided estimate of the square function in terms of boundary data and 
the nontangential maximal function. We shall only work in the case $\Omega=\BBR^n_+$ with a coefficient tensor satisfying $A_{0j}^{\alpha\beta}=\delta_{\alpha\beta}\delta_{0j}$ since by section 2 of \cite{DHM} the problem can alway be reduced to this case. We shall return to this point in sections 6-8 where we outline proofs of our three main results.

For brevity of our argument we will perform below a common calculation which we then apply in different settings. The first case we will consider is the perturbation case. Let 
\begin{equation}\label{ES-new}
\mathcal{L}u=\left[ \partial_{i} \left(A_{ij}^{\alpha \beta}(x) \partial_{j} u_{\beta}\right)
+B_{i}^{\alpha \beta}(x) \partial_{i}u_{\beta}\right]_{\alpha}
\end{equation}
be an operator that satisfies $A_{0j}^{\alpha\beta}=\delta_{\alpha\beta}\delta_{0j}$ and the Carleson condition \eqref{Car_hatAAxxx}.

Consider an another operator 
\begin{equation}\label{ES-new2}
\mathcal{L}_1u=\left[ \partial_{i} \left(\bar{A}_{ij}^{\alpha \beta}(x) \partial_{j} u_{\beta}\right)
+\bar{B}_{i}^{\alpha \beta}(x) \partial_{i}u_{\beta}\right]_{\alpha}
\end{equation}

for which the only information we have is that its coefficients are close to ${\mathcal L}$ in the following sense that
\begin{equation}\label{Car-m}
dm(x)=\left[\left(\sup_{B_{\delta(x)/2}(x)}|A-\bar{A}|\right)^{2}\delta^{-1}(x)
+\left(\sup_{B_{\delta(x)/2}(x)}|B-\bar{B}|\right)^{2}\delta(x)\right] \,dx
\end{equation}
is a Carleson measure with small Carleson norm $\|m\|_{\mathcal C}$. It follows that if we denote by
$$\varepsilon_{ij}^{\alpha\beta}={A}_{ij}^{\alpha \beta}-\bar{A}_{ij}^{\alpha \beta},\qquad b_i^{\alpha\beta}=
{B}_{i}^{\alpha \beta}-\bar{B}_{i}^{\alpha \beta},$$
then any solution of ${\mathcal L}_1u=0$ can be written as
\begin{equation}\label{eq-sys-pert}
[{\mathcal L}u]_\alpha=\left[\partial_i(\varepsilon_{ij}^{\alpha\beta}\partial_ju_\beta)+b_i^{\alpha\beta}\partial_iu_\beta\right]_\alpha.
\end{equation}

The second case we consider is as follows. Let $u$ to be a solution to ${\mathcal L}u=0$ where ${\mathcal L}$ is as in \eqref{ES-new} again satisfying $A_{0j}^{\alpha\beta}=\delta_{\alpha\beta}\delta_{0j}$ and the Carleson condition \eqref{Car_hatAAxxx}.
Denote by $w^k_\alpha=\partial_k u_\alpha$ for some fixed $k=0,1,\dots, n-1$. If follows that $w^k=(w^k_\alpha)$
solves the following elliptic system
\begin{equation}\label{eq-sys-deriv}
[{\mathcal L}w^k]_\alpha=\left[\partial_i\left(\partial_k({A}_{ij}^{\alpha \beta})w^j_\beta\right)+\partial_k(B_i^{\alpha\beta})w^i_\beta\right]_\alpha.
\end{equation}

Observe that both \eqref{eq-sys-pert} and \eqref{eq-sys-deriv} have similar structures and hence can be dealt with (mostly) in parallel. In what follows the vector $v$ either represents $u$ as in \eqref{eq-sys-pert} or alternatively it is one component $w^k$ for some $k$ fixed. Let us denote by ${\mathcal L}_0$ the second order part of the operator \eqref{ES-new}, that is
\begin{equation}\label{ES-l0}
\mathcal{L}_0v=\left[ \partial_{i} \left(A_{ij}^{\alpha \beta}(x) \partial_{j} v_{\beta}\right)\right]_{\alpha}.
\end{equation}

\vglue1mm

Fix an arbitrary $y'\in\partial\Omega\equiv{\mathbb{R}}^{n-1}$ and pick 
a smooth cutoff function $\zeta$ which is $x_0-$independent and satisfies
\begin{equation}\label{cutoff-F}
\zeta= 
\begin{cases}
1 & \text{ in } B_{r}(y'), 
\\
0 & \text{ outside } B_{2r}(y').
\end{cases}
\end{equation}
Moreover, assume that $r|\nabla \zeta| \leq c$ for some positive constant $c$ independent of $y'$. 
We begin by considering the integral quantity 
\begin{equation}\label{A00}
\mathcal{I}:=\iint_{[0,r]\times B_{2r}(y')}A_{ij}^{\alpha\beta}\partial_{j}v_{\beta} 
\partial_{i}v_{\alpha}x_0\zeta\,dx'\,dx_0
\end{equation}
with the usual summation convention understood. In relation to this we note that the uniform 
ellipticity \eqref{EllipA} gives 
\begin{equation}\label{cutoff-AA}
\mathcal{I}\geq{\lambda}\iint_{[0,r]\times B_{2r}}\sum_{\alpha}|\nabla v_{\alpha}|^2 x_0\zeta\,dx'\,dx_0
={\lambda}\iint_{[0,r]\times B_{2r}}|\nabla v|^2 x_0\zeta\,dx'\,dx_0,
\end{equation}
where we agree henceforth to abbreviate $B_{2r}:=B_{2r}(y')$ whenever convenient. 
The idea now is to integrate by parts the formula for $\mathcal I$ in order 
to relocate the $\partial_i$ derivative. This gives 
\begin{align}\label{I+...+IV}
\mathcal{I}
&= \int_{\partial\left[(0,r)\times B_{2r}\right]} 
A_{ij}^{\alpha\beta}\partial_{j}v_{\beta}v_{\alpha}x_0\zeta\nu_{x_i}\,d\sigma 
\nonumber\\[4pt]
&\quad -\iint_{[0,r]\times B_{2r}}({\mathcal L}_0v\cdot v)\, x_0\zeta\,dx'\,dx_0 
\nonumber\\[4pt]
&\quad -\iint_{[0,r]\times B_{2r}}A_{ij}^{\alpha\beta}\partial_{j}v_{\beta}v_{\alpha}\partial_{i}x_0\zeta\,dx'\,dx_0 
\nonumber\\[4pt]
&\quad -\iint_{[0,r]\times B_{2r}}A_{ij}^{\alpha\beta}\partial_{j}v_{\beta}v_{\alpha}x_0\partial_{i}\zeta\,dx'\,dx_0
\nonumber\\[4pt]
&=:I+II+III+IV,
\end{align}
where $\nu$ is the outer unit normal vector to $(0,r)\times B_{2r}(y')$. 
Bearing in mind $A_{0j}=0_{N\times N}$ and upon recalling that we are assuming $A_{00}=I_{N\times N}$, 
the boundary term $I$ simply becomes
\begin{equation}\label{cutoff-BBB}
I=\int_{B_{2r}}\partial_{0}v_{\beta}(r,x')v_{\beta}(r,x')\,r\,\zeta\,dx'.
\end{equation}

As $\partial_ix_0=0$ for $i>0$ the term $III$ is non-vanishing only for $i=0$. We further split this term
by considering the cases when $j=0$ and $j>0$. When $j=0$, we use that $A_{00}^{\alpha\beta}=I_{N\times N}$. 
This yields 
\begin{align}\label{u6fF}
III_{\{j=0\}} &=-\frac{1}{2}\iint_{[0,r]\times B_{2r}} 
\sum_\beta\partial_{0}\left(v_{\beta}^{2}\zeta\right)\,dx'\,dx_0 
\nonumber\\[4pt]
&=-\frac{1}{2}\int_{B_{2r}}\sum_\beta v_{\beta}(r,x')^{2}\zeta\,dx'
+\frac{1}{2}\int_{B_{2r}}\sum_\beta v_{\beta}(0,x')^{2}\zeta\,dx'.
\end{align}
Corresponding to $j>0$ we simply recall $A_{0j}=0_{N\times N}$ to conclude that $III_{\{j>0\}}=0$. 

Since $\mathcal L-{\mathcal L_0}$ gives us the first order term we can write $II$ as

\begin{equation}\label{cutoff-CCC}
II=\iint_{[0,r]\times B_{2r}}B_{i}^{\alpha\beta}(\partial_{i}v_{\beta})v_{\alpha}x_0\zeta\,dx'\,dx_0 -\iint_{[0,r]\times B_{2r}}({\mathcal L}v\cdot v)\, x_0\zeta\,dx'\,dx_0 .
\end{equation}
To further estimate the first term we use Cauchy-Schwarz inequality, the Carleson condition for $B$ and 
Theorem~\ref{T:Car} in order to write
\begin{align}\label{TWO-TWO}
II_1 &\leq\left(\iint_{[0,r]\times B_{2r}}\left(B_i^{\alpha\beta}\right)^{2} 
|v_{\alpha}|^{2} x_0\zeta\,dx'\,dx_0\right)^{1/2}  
\cdot\left(\iint_{[0,r]\times B_{2r}}|\partial_{j}v_{\beta}|^{2}x_0\zeta\,dx'\,dx_0\right)^{1/2} 
\nonumber\\[4pt]
&\leq C(\lambda,\Lambda,N)\left(\|\mu\|_{\mathcal{C}}\int_{B_{2r}} 
\left[\tilde{N}^r_a(v)\right]^{2}\,dx'\right)^{1/2}\cdot\mathcal{I}^{1/2}. 
\end{align}

We add up all terms we have so far to obtain
\begin{align}\label{square01}
\mathcal{I} &\leq\int_{B_{2r}}\partial_{0}v_{\beta}(r,x')v_{\beta}(r,x')\,r\,\zeta\,dx' + II_2
\nonumber\\[4pt]
&\quad-\frac{1}{2}\int_{B_{2r}}\sum_\beta v_{\beta}(r,x')^{2}\zeta\,dx'
+\frac{1}{2}\int_{B_{2r}}\sum_\beta v_{\beta}(0,x')^{2}\zeta\,dx' 
\nonumber\\[4pt]
&\quad +C(\lambda,\Lambda,n,N)\|\mu\|_{\mathcal{C}}\int_{B_{2r}}\left[\tilde{N}^{r}_a(v)\right]^2\,dx' 
+\frac14\mathcal{I}+IV,
\end{align}
where we have used the arithmetic-geometric inequality for expression bounding the term $II_1$ in \eqref{TWO-TWO}. Here
\begin{equation}\label{eq-2-1}
II_2=-\iint_{[0,r]\times B_{2r}}({\mathcal L}v\cdot v)\, x_0\zeta\,dx'\,dx_0.
\end{equation}
\vglue1mm

Consider now the situation where ${\mathcal L}v$ is as in \eqref{eq-sys-pert}. We can then further estimate
the term $II_2$. Using the righthand side of  \eqref{eq-sys-pert} we have:
\begin{align}\label{eq-2-2}
II_2=& -\iint_{[0,r]\times B_{2r}}(\partial_i(\varepsilon_{ij}^{\alpha\beta}\partial_jv_\beta)+b_i^{\alpha\beta}\partial_iv_\beta) v_\alpha\, x_0\zeta\,dx'\,dx_0\\\nonumber
=&\quad\iint_{[0,r]\times B_{2r}}\varepsilon_{ij}^{\alpha\beta}\partial_jv_\beta \partial_iv_\alpha\, x_0\zeta\,dx'\,dx_0\\\nonumber
+&\quad\iint_{[0,r]\times B_{2r}}\varepsilon_{0j}^{\alpha\beta}\partial_jv_\beta v_\alpha\, \zeta\,dx'\,dx_0\\\nonumber
+&\quad\iint_{[0,r]\times B_{2r}}\varepsilon_{ij}^{\alpha\beta}\partial_jv_\beta v_\alpha\, x_0\partial_i\zeta\,dx'\,dx_0\\\nonumber
+&\quad\int_{B_{2r}}\varepsilon_{0j}^{\alpha\beta}\partial_jv_\beta v_\alpha\,r\,\zeta\,dx'\\\nonumber
-&\quad\iint_{[0,r]\times B_{2r}}b_i^{\alpha\beta}\partial_iv_\beta v_\alpha\, x_0\zeta\,dx'\,dx_0=II_{21}+II_{22}+II_{23}+II_{24}+II_{25}.
\end{align}

The term $II_{21}$ (using the Cauchy-Schwarz) is bounded by the square function of $v$, using \eqref{Car-m} we have
$|\varepsilon_{ij}^{\alpha\beta}|\lesssim\|m\|^{1/2}_{\mathcal C}$. For terms $II_{22}$ and $II_{25}$ we have estimates very similar to \eqref{TWO-TWO}. This gives using the AG-inequality
\begin{align}\label{eq-2-3}
|III_{21}|&+|II_{22}|+|II_{25}|\le \frac14{\mathcal I} + C\|m\|^{1/2}_{\mathcal{C}}{\mathcal I} \\\nonumber
&+\,C\|m\|_{\mathcal{C}}\int_{B_{2r}}\left[\tilde{N}^{r}_a(v)\right]^2\,dx' . 
\end{align}
It follows using \eqref{square01}
\begin{align}\label{square01-alt1}
\left(\frac12-C\|m\|^{1/2}_{\mathcal{C}}\right)\mathcal{I} &\leq\int_{B_{2r}}\partial_{0}v_{\beta}(r,x')v_{\beta}(r,x')\,r\,\zeta\,dx' +\int_{B_{2r}}\varepsilon_{0j}^{\alpha\beta}\partial_jv_\beta v_\alpha\,r\,\zeta\,dx'
\nonumber\\[4pt]
&\quad-\frac{1}{2}\int_{B_{2r}}\sum_\beta v_{\beta}(r,x')^{2}\zeta\,dx'
+\frac{1}{2}\int_{B_{2r}}\sum_\beta v_{\beta}(0,x')^{2}\zeta\,dx' 
\nonumber\\[4pt]
&\quad +C(\lambda,\Lambda,n,N)(\|\mu\|_{\mathcal{C}}+\|m\|_{\mathcal{C}})\int_{B_{2r}}\left[\tilde{N}^{r}_a(v)\right]^2\,dx' 
+II_{23}+IV,
\end{align}
Clearly, we shall require $\frac12-C\|m\|^{1/2}_{\mathcal{C}}>0$ so that righthand side is positive.

Consider now a sequence of disjoint boundary balls 
$(B_r(y'_k))_{k\in\mathbb N}$ such that $\cup_{k}B_{2r}(y'_k) $ covers $\partial\Omega={\BBR}^{n-1}$ with finite overlap
and consider a partition of unity $(\zeta_{k})_{k\in\mathbb N}$ subordinate to this cover. That is, 
assume $\sum_k \zeta_{k} = 1$ on ${\BBR}^{n-1}$ and each $\zeta_{k}$ is supported in $B_{2r}(y'_k)$. 
Consider \eqref{square01-alt1} on each such ball and we sum over all $k$.  Given that $\sum_k \partial_i\zeta_{k} = 0$ for each $i$, by summing 
\eqref{square01-alt1} over all $k$'s gives $\sum_{k} II_{23}+IV= 0$. It follows that for $\|m\|_{\mathcal C}$ small we have
\begin{align}\label{square02}
&\hskip -0.20in 
C(\lambda)\iint_{[0,r]\times{\BBR}^{n-1}}|\nabla v|^2\,x_0\,dx'\,dx_0 
\nonumber\\[4pt]
&\hskip 0.20in
\leq\int_{{\BBR}^{n-1}}\partial_{0}v_{\beta}(r,x')v_{\beta}(r,x')\,r\,dx' +\int_{B_{2r}}\varepsilon_{0j}^{\alpha\beta}\partial_jv_\beta v_\alpha\,r\,\zeta\,dx'
\nonumber\\[4pt]
&
\quad -\frac{1}{2}\int_{{\BBR}^{n-1}}\sum_\beta v_{\beta}(r,x')^{2}\,dx'
+\frac{1}{2}\int_{{\BBR}^{n-1}}\sum_\beta v_{\beta}(0,x')^{2}\,dx' 
\nonumber\\[4pt]
&\quad +C(\|\mu\|_{\mathcal{C}}+\|m\|_{\mathcal C})\int_{{\BBR}^{n-1}}\left[\tilde{N}^{r}_a(v)\right]^2\,dx'.
\end{align}
We integrate \eqref{square02} in $r$ over 
$[0,r']$ and then divide by $r'$. This gives (after relabelling $r'$ back to $r$)
\begin{align}\label{square02-alt1}
&\hskip -0.20in 
C(\lambda)\iint_{[0,r]\times{\BBR}^{n-1}}|\nabla v|^2\,x_0(1-\textstyle\frac{t}r)\,dx'\,dx_0 \displaystyle
+\frac2r\iint_{[0,r]\times \partial\Omega}|v|^2dx'\,dx_0
\nonumber\\[4pt]
&\hskip 0.20in
\leq \int_{{\BBR}^{n-1}}|v|^2(r,x')\,dx'
+\int_{{\BBR}^{n-1}}|v|^2(0,x')\,dx' 
\nonumber\\[4pt]
&\qquad +C(\|\mu\|_{\mathcal{C}}+\|m\|_{\mathcal C})\int_{{\BBR}^{n-1}}\left[\tilde{N}^{r}_a(v)\right]^2\,dx'.
\end{align}
We note that the second term on the righthand side of \eqref{square02} was estimated using the Cauchy-Schwarz and the interior Cacciopoli inequality.

We use \eqref{square01-alt1} for another local estimate. Instead of summing over different balls $B_r$ covering $\mathbb R^{n-1}$ we estimate the terms $II_{23}$ and $IV$. Both of these terms are of the same type and can be bounded (up to a constant) by
\begin{equation}\label{eq5.16}
\iint_{[0,r]\times B_{2r}}|\nabla v||v|x_0|\partial_T\zeta|dx'dx_0,
\end{equation}
where $\partial_T\zeta$ denotes any of the derivatives in the direction parallel to the boundary. Recall that $\zeta$ is a smooth cutoff function equal to $1$ on $B_r$ and $0$ outside $B_{2r}$. In particular, we may assume $\zeta$ to be of the form $\zeta=\eta^2$ for another smooth function $\eta$ such that $|\nabla_T\eta|\le C/r$. By Cauchy-Schwartz \eqref{eq5.16} can be further estimated by
\begin{equation}\label{eq5.17}
\left(\iint_{[0,r]\times B_{2r}}|\nabla v|^2x_0(\eta)^2dx'dx_0\right)^{1/2}\left(\iint_{[0,r]\times B_{2r}}|v|^{2}x_0|\nabla_T\eta|^2dx'dx_0\right)^{1/2}
\end{equation}
\begin{equation}
\lesssim{\mathcal I}^{1/2}\left(\frac1r\iint_{[0,r]\times B_{2r}}|v|^2dx'dx_0\right)^{1/2}\le \varepsilon{\mathcal I}+C_\varepsilon\int_{B_{2r}}\left[\tilde{N}^r_{a}(v)\right]^{2}\,dx'.\nonumber
\end{equation}
In the last step we have used AG-inequality and a trivial estimate of the solid integral $|v|^2$ by the averaged nontangential maximal function. Substituting \eqref{eq5.17} into \eqref{square01-alt1} and again integrating in $r$ over $[0,r']$ and dividing by $r'$ exactly as above finally yields the following lemma.

\begin{lemma}\label{S3:L4-alt1}
Let $\Omega=\BBR^n_+$ and assume that 
$A$ is strongly elliptic, satisfies $A_{0j}^{\alpha\beta}=\delta_{\alpha\beta}\delta_{0j}$,
and the measure $\mu$ defined as in \eqref{Car_hatAAxxx} is Carleson. Let ${\mathcal L}_1$ be another operator such that
\begin{equation}\label{ES-new2a}
\mathcal{L}_1u=\left[ \partial_{i} \left(\bar{A}_{ij}^{\alpha \beta}(x) \partial_{j} u_{\beta}\right)
+\bar{B}_{i}^{\alpha \beta}(x) \partial_{i}u_{\beta}\right]_{\alpha}
\end{equation}
where
\begin{equation}\label{Car-m2}
dm(x)=\left[\left(\sup_{B_{\delta(x)/2}(x)}|A-\bar{A}|\right)^{2}\delta^{-1}(x)
+\left(\sup_{B_{\delta(x)/2}(x)}|B-\bar{B}|\right)^{2}\delta(x)\right] \,dx
\end{equation}
is a Carleson measure with small Carleson norm $\|m\|_{\mathcal C}$.

Then there exists a constant $C=C(n,N,\lambda,\Lambda)$ such that for all $r>0$ and any energy solution of ${\mathcal L}_1u=0$ we have
\begin{align}\label{S3:L4:E00}
&
C(\lambda)\iint_{[0,r/2]\times\partial\Omega}|\nabla u|^{2}x_0\,dx'\,d x_0 
+\frac{2}{r}\iint_{[0,r]\times\partial \Omega} |u(x_0,x')|^{2}\,dx'\,dx_0 
\nonumber\\[4pt]
&
\leq\int_{\partial\Omega}|u(0,x')|^{2}\,dx' 
+\int_{\partial\Omega}|u(r,x')|^{2}\,dx'
+C(\|\mu\|_{\mathcal{C}}+\|m\|_{\mathcal{C}})\int_{\partial\Omega}\left[\tilde{N}^{r}_a(u)\right]^{2}\,dx'.
\end{align}
Under the same assumptions we also have for any $r>0$
\begin{eqnarray}\label{eq5.15-alt1}
&&\iint_{[0,r/2]\times B_r}|\nabla u|^{2}x_0\,dx'\,d x_0  \\\nonumber
&\leq&C\Big[\int_{B_{2r}}|u(0,x')|^{2}\,dx' 
+\int_{B_{2r}}|u(r,x')|^{2}\,dx'\\\nonumber
&&+(\|\mu\|_{\mathcal{C}}+\|m\|_{\mathcal{C}})\int_{B_{2r}}\left[\tilde{N}^r_{a}(u)\right]^{2}\,dx'\Big]\\\nonumber
&\leq& C(2+\|\mu\|_{\mathcal{C}}+\|m\|_{\mathcal{C}})\int_{B_{2r}}\left[\tilde{N}^{2r}_{a}(u)\right]^{2}\,dx'.
\end{eqnarray}
\end{lemma}

\begin{corollary}\label{S4:C1-alt1} 
We retain the assumptions of Lemma~\ref{S3:L4-alt1}. Given any weak solution $u$ of ${\mathcal L}_1u=0$,
we have the estimate
\begin{equation}\label{Eqqq-3-alt1}
\|S_a(u)\|_{L^2(\partial\Omega)}\leq C\|\tilde{N}_a(u)\|_{L^2(\partial \Omega)}.
\end{equation}
\end{corollary}
This can be seen by taking $r\to\infty$ in the previous Lemma.\vglue2mm

We now return to \eqref{square01} and consider the second case when $v$ is one of  the vectors $w^k$
given by \eqref{eq-sys-deriv}. The calculation below is similar to \cite{DPR} where the scalar case $N=1$ was established. We proceed again to estimate the term $II_2$ for $k>0$. This gives (using the fact that
$\partial_k {A}_{0j}^{\alpha \beta}=0$):
\begin{align}\label{eq-2-2-alt2}
II_2=& -\iint_{[0,r]\times B_{2r}}\left[\partial_i\left(\partial_k({A}_{ij}^{\alpha \beta})w^j_\beta\right)+\partial_k(B_i^{\alpha\beta})w^i_\beta\right] w^k_\alpha\, x_0\zeta\,dx'\,dx_0\\\nonumber
=&\quad\iint_{[0,r]\times B_{2r}}\partial_k({A}_{ij}^{\alpha \beta})w^j_\beta \partial_iw^k_\alpha\, x_0\zeta\,dx'\,dx_0\\\nonumber
+&\quad\iint_{[0,r]\times B_{2r}}\partial_k({A}_{ij}^{\alpha \beta})w^j_\beta w^k_\alpha\, x_0\partial_i\zeta\,dx'\,dx_0\\\nonumber
+&\quad \iint_{[0,r]\times B_{2r}}B_i^{\alpha\beta}w^i_\beta  w^k_\alpha\, x_0\partial_k\zeta\,dx'\,dx_0\\\nonumber
+&\quad \iint_{[0,r]\times B_{2r}}B_i^{\alpha\beta}\partial_iw^k_\beta w^k_\alpha\, x_0\zeta\,dx'\,dx_0\\\nonumber
+&\quad \iint_{[0,r]\times B_{2r}}B_i^{\alpha\beta}w^i_\beta \partial_k w^k_\alpha\, x_0\zeta\,dx'\,dx_0=II_{21}+II_{22}+II_{23}+II_{24}+II_{25}.
\end{align}
Terms $II_{21}$, $II_{24}$ and $II_{25}$ have estimates similar to \eqref{TWO-TWO}, that is
$$|II_{21}|+|II_{24}|+|II_{25}|\le C(\lambda,\Lambda,N)\left(\|\mu\|_{\mathcal{C}}\int_{B_{2r}} 
\left[\tilde{N}^r_a(w)\right]^{2}\,dx'\right)^{1/2}\cdot\mathcal{I}^{1/2}. $$
It follows we have an analogue of \eqref{square01-alt1}, namely
\begin{align}\label{square01-alt2}
\frac14\mathcal{I} &\leq\int_{B_{2r}}\partial_{0}w^k_{\beta}(r,x')w^k_{\beta}(r,x')\,r\,\zeta\,dx' \nonumber\\[4pt]
&\quad-\frac{1}{2}\int_{B_{2r}}\sum_\beta w^k_{\beta}(r,x')^{2}\zeta\,dx'
+\frac{1}{2}\int_{B_{2r}}\sum_\beta w^k_{\beta}(0,x')^{2}\zeta\,dx' 
\nonumber\\[4pt]
&\quad +C(\lambda,\Lambda,n,N)\|\mu\|_{\mathcal{C}}\int_{B_{2r}}\left[\tilde{N}^{r}_a(\nabla u)\right]^2\,dx' 
+II_{22}+II_{23}+IV,
\end{align}
where $w^k=\partial_k u$ for $k=1,2,\dots,n-1$. As before we then integrate over $[0,r']$ and dive by by $r'$. See the resulting lemma below.

We also want to establish a local result that follows from \eqref{square01-alt2}. Term $IV$ has the same estimate as before. Both terms $II_{22}$ and $II_{23}$ are of the same type and we use $|\nabla A|,\,|B|\lesssim\|\mu\|^{1/2}_{\mathcal C}/x_0$. The remainder has a trivial estimate by $\int_{B_{2r}}\left[\tilde{N}^{r}_a(\nabla u)\right]^2\,dx'$.  We get the following.

\begin{lemma}\label{S3:L4-alt2}
Let $\Omega=\BBR^n_+$ and let 
\begin{equation}\label{ES-old2a}
\mathcal{L}u=\left[ \partial_{i} \left({A}_{ij}^{\alpha \beta}(x) \partial_{j} u_{\beta}\right)
+{B}_{i}^{\alpha \beta}(x) \partial_{i}u_{\beta}\right]_{\alpha}
\end{equation}
has coefficients $A$ that are strongly elliptic, satisfy $A_{0j}^{\alpha\beta}=\delta_{\alpha\beta}\delta_{0j}$,
and the measure $\mu$ defined as in \eqref{Car_hatAAxxx} is Carleson.

Then there exists a constant $C=C(n,N,\lambda,\Lambda)$ such that for all $r>0$ and any energy solution ${\mathcal L}u=0$ we have
\begin{align}\label{S3:L4:E00-alt2}
&
\frac{\lambda}2\iint_{[0,r/2]\times\partial\Omega}|\nabla^2 u|^{2}x_0\,dx'\,d x_0 
+\frac{2}{r}\iint_{[0,r]\times\partial \Omega} |\nabla_T u(x_0,x')|^{2}\,dx'\,dx_0 
\nonumber\\[4pt]
&
\leq\int_{\partial\Omega}|\nabla_T u(0,x')|^{2}\,dx' 
+\int_{\partial\Omega}|\nabla_T u(r,x')|^{2}\,dx'
+C\|\mu\|_{\mathcal{C}}\int_{\partial\Omega}\left[\tilde{N}^{r}_a(\nabla u)\right]^{2}\,dx'.
\end{align}
Under the same assumptions we also have for any $r>0$
\begin{eqnarray}\label{eq5.15-alt2}
&&\iint_{[0,r/2]\times B_r}|\nabla^2 u|^{2}x_0\,dx'\,d x_0  \\\nonumber
&\leq&C\Big[\int_{B_{2r}}|\nabla_T u(0,x')|^{2}\,dx' 
+\int_{B_{2r}}|\nabla_T u(r,x')|^{2}\,dx'+\|\mu\|_{\mathcal{C}}\int_{B_{2r}}\left[\tilde{N}^r_{a}(\nabla u)\right]^{2}\,dx'\Big]\\\nonumber
&\leq& C(2+\|\mu\|_{\mathcal{C}})\int_{B_{2r}}\left[\tilde{N}^{2r}_{a}(\nabla u)\right]^{2}\,dx'.
\end{eqnarray}
\end{lemma}

\begin{proof} After summing \eqref{square01-alt2} over all partitions of unity, over $k=1,2,\dots n-1$ and integrating in $r$ we only obtain control of 
$$\iint_{[0,r/2]\times B_r}|\nabla(\nabla_T u)|^{2}x_0\,dx'\,d x_0$$
on the lefthand side of \eqref{S3:L4:E00-alt2} and \eqref{eq5.15-alt2} as we have omitted the $w^0$ term. But clearly,
$$|\nabla \partial_0 u|\lesssim |\nabla \nabla_T u|+|\partial^2_{0} u|,$$
and therefore it only remains to establish a bound for
$$\iint_{[0,r/2]\times B_r}|\partial^2_0 u|^{2}x_0\,dx'\,d x_0.$$
We do it using the PDE for $u$. Since ${\mathcal L}u=0$ and $A_{00}^{\alpha\beta}=I_{N\times N}$ we have
$$\partial^2_0 u_\alpha=-\sum_{(i,j)\ne (0,0)}A_{ij}^{\alpha\beta}\partial_i\partial_ju_\beta-\partial_i(A_{ij}^{\alpha\beta})w^j_\beta-B_i^{\alpha\beta}w^i_\beta.$$
It follows
\begin{align}
\iint_{[0,r/2]\times B_r}|\partial^2_0 u|^{2}x_0\,dx'\,d x_0&\lesssim \iint_{[0,r/2]\times B_r}|\nabla\nabla_T u|^{2}x_0\,dx'\,d x_0\\&+\iint_{[0,r/2]\times B_r}|\nabla A|^2|\nabla u|^2x_0x_0\,dx'\,d x_0.
\nonumber
\end{align}
The last term can be estimated using the Carleson condition by
$$\|\mu\|_{\mathcal{C}}\int_{B_{2r}}\left[\tilde{N}^r_{a}(\nabla u)\right]^{2}\,dx'.$$
From this both estimates follow. After taking limit $r\to\infty$ we get:
\end{proof}

\begin{corollary}\label{S4:C1-alt2} 
We retain the assumptions of Lemma~\ref{S3:L4-alt2}. Given any weak solution $u$ of ${\mathcal L}u=0$,
we have the estimate
\begin{equation}\label{Eqqq-3-alt2}
\|S_a(\nabla u)\|_{L^2(\partial\Omega)}\leq C\|\tilde{N}_a(\nabla u)\|_{L^2(\partial \Omega)}.
\end{equation}
\end{corollary}

\section{Bounds for the nontangential maximal function by the square function}
\label{SS:43}

Our aim in this section is to establish a reverse versions of the inequalities \eqref{Eqqq-3-alt1} and \eqref{Eqqq-3-alt2}. Like the previous section we shall assume that $\Omega={\mathbb R}^{n}_+$. 
 
The approach necessarily differs from the usual argument in the scalar elliptic case due to the fact 
that certain estimates, such as interior H\"older regularity of a weak solution, are unavailable for 
the class of systems presently considered. Hence, alternative arguments bypassing such difficulties 
must be devised. This has been done in \cite{DHM} and it will allow us to shorten the arguments here substantially as we can use many results from \cite{DHM} without significant changes. Where we diverge substantially from the approach given in \cite{DHM} we present a complete alternative argument.

Again the initial part of our approach is generic and applies to both cases we consider in this paper. Assume therefore that $v\in W^{1,2}_{loc}(\Omega;\mathbb R^N)$ with some decay as $x_0\to\infty$
and let $w$ be the $L^2$ average of $v$ that is
\begin{equation}\label{ww}
w(x):=\left(\dint_{B_{\delta(x)/2}(x)}|v|^{2}(z)\,dz\right)^{1/2}.
\end{equation}

For the function $w$ defined in $\Omega$ as in \eqref{w}, and a constant $\nu>0$, 
define the set
\begin{equation}\label{E}
E_{\nu,a}:=\big\{x'\in\partial\Omega:\,N_{a}(w)(x')>\nu\big\}
\end{equation}
(where, as usual, $a>0$ is a fixed background parameter), 
and consider the map $h:\partial\Omega\to\BBR$ given at each $x'\in\partial\Omega$ by 
\begin{equation}\label{h}
h_{\nu,a}(w)(x'):=\inf\left\{x_0>\phi(x'):\,\sup_{z\in\Gamma_{a}(x_0,x')}w(z)<\nu\right\}
\end{equation}
with the convention that $\inf\varnothing=\infty$. 
At this point it is not clear whether $h_{\nu,a}(w,x')<\infty$ for all points $x'\in\partial\Omega$, which is why we have assumed some decay of $v$ as $x_0\to\infty$. Both cases we consider will satisfy this. A shown in \cite{DHM} we have the following.

\begin{lemma}\label{S3:L5} Fix two positive numbers $\nu,a$. Then the following holds.
\vglue2mm

\noindent (i)
The function $h_{\nu, a}(w)$ is Lipschitz, with a Lipschitz constant $1/a$. That is,
\begin{equation}\label{Eqqq-5}
\left|h_{\nu,a}(w)(x')-h_{\nu,a}(w)(y')\right|\leq a^{-1}|x'-y'|
\end{equation}
for all $x',y'\in\partial\Omega$.

\vglue2mm

\noindent (ii)
Given an arbitrary $x'\in E_{\nu,a}$, let $x_0:=h_{\nu,a}(w)(x')$. Then there exists a 
point $y=(y_0,y')\in\partial\Gamma_{a}(x_0,x')$ such that $w(y)=\nu$ and $h_{\nu,a}(w)(y')=y_0$. 		
\end{lemma}

\begin{lemma}\label{l6} Let $v$ and $w$ are as before. 
For any $a>0$ there exists $b=b(a)>0$ and $\gamma=\gamma(a)>0$ such that the following holds. 
Having fixed an arbitrary $\nu>0$, for each point $x'$ from the set 
\begin{equation}\label{Eqqq-17}
\big\{x':\,N_{a}(w)(x')>\nu\mbox{ and }S_{b}(v)(x')\leq\gamma\nu\big\}
\end{equation}
there exists a boundary ball $R$ with $x'\in 2R$ and such that
\begin{equation}\label{Eqqq-18}
\big|w\big(h_{\nu,a}(w)(z'),z'\big)\big|>\nu/{2}\,\,\text{ for all }\,\,z'\in R.
\end{equation}
\end{lemma}

Given a Lipschitz function $h:{\mathbb{R}}^{n-1}\to{\mathbb{R}}$, denote by $M_h$ the 
Hardy-Littlewood maximal function considered on the graph of $h$. That is, 
given any locally integrable function $f$ on the Lipschitz surface 
$\Lambda_h=\{(h(z'),z'):\,z'\in\BBR^{n-1}\}$, define 
$(M_h f)(x):=\sup_{r>0}\dint_{\Lambda_h\cap B_r(x)}|f|\,d\sigma$ for each $x\in\Lambda_h$. 

\begin{corollary}\label{S3:L6} Let $v$, $w$ be as above. For a fixed $a>0$, 
consider $b,\,\gamma$ be as in Lemma~\ref{l6}. Then there exists a finite 
constant $C=C(n)>0$ with the property that for any $\nu>0$ and any point $x'\in E_{\nu,a}$ 
such that $S_{b}(v)(x')\leq\gamma\nu$ one has
\begin{equation}\label{Eqqq-23}
(M_{h_{\nu,a}}w)\big(h_{\nu,a}(x'),x'\big)\geq\,C\nu.
\end{equation}
\end{corollary}

The next Lemma is an analogue of \cite[Lemma 5.4]{DHM}. 
Motivated by \cite{FMZ} we depart in the proof from the approach given in \cite{DHM}. This allows us to  drop some assumptions on the coefficients of our operator we have made in \cite{DHM} which are not necessary.

\begin{lemma}\label{S3:L8-alt1} 
Let $\Omega={\mathbb R}^n_+$ and let ${\mathcal L}$ be an operator of the form
\begin{equation}\label{ES-new2ax}
\mathcal{L}u=\left[ \partial_{i} \left({A}_{ij}^{\alpha \beta}(x) \partial_{j} u_{\beta}\right)+\partial_{0} \left({\varepsilon}_{0j}^{\alpha \beta}(x) \partial_{j} u_{\beta}\right)
+{B}_{i}^{\alpha \beta}(x) \partial_{i}u_{\beta}\right]_{\alpha}
\end{equation}
such that $A_{0j}^{\alpha\beta}(x)=\delta_{\alpha\beta}\delta_{0j}$ and
\begin{equation}\label{Car-m22}
d\mu(x)=\left[\left(\sup_{B_{\delta(x)/2}(x)}|\varepsilon|\right)^{2}\delta^{-1}(x)
+\left(\sup_{B_{\delta(x)/2}(x)}|B|\right)^{2}\delta(x)\right] \,dx
\end{equation}
is a Caleson measure with norm $M$. 

\noindent Suppose $u$ is a weak solution of $\mathcal{L}u=0$ in $\Omega$.  For a fixed $a>0$, consider an arbitrary Lipschitz function $\hbar:{\mathbb R}^{n-1}\to \mathbb R$ such that
\begin{equation}
\|\nabla \hbar\|_{L^\infty}\le 1/a,\qquad \hbar(x')\ge 0\text{ for all }x'\in{\mathbb R}^{n-1}.\label{hbarprop}
\end{equation}
Then for $b=b(6a)>0$ as in Lemma~\ref{l6} we have the following. For an arbitrary surface ball $\Delta_r=B_r(Q)\cap\partial\Omega$, with $Q\in\partial\Omega$ and $r>0$ such that at at least one point of $\Delta_r$
the inequality $\hbar(x')\le 2r$ holds we have the following estimate:
\begin{align}\label{TTBBMM}
\int_{1/6}^6\int_{\Delta_r}\big|u\big(\theta\hbar(x'),x'\big)\big|^2\,dx'\,d\theta
&\leq C\|S_b(u)\|_{L^2(\Delta_{2r})}
\|\tilde{N}_a(u)\|_{L^2(\Delta_{2r})}
\nonumber\\
&\quad+C\|S_b(u)\|^2_{L^2(\Delta_{2r})}+\frac{C}{r}\iint_{\mathcal{K}}|u|^{2}\,dX,
\end{align}
for some $C\in(0,\infty)$ that only depends on $a,\Lambda,n,N$ and $M$ but not on $u$ or $\Delta_r$. By $\mathcal{K}$ we have denoted a region inside $\Omega$ such that its diameter, 
distance to the graph $(\hbar(\cdot),\cdot)$, and distance to $Q$, are all comparable to $r$. 
Also, the cones used to define the square and nontangential 
maximal functions in this lemma have vertices on $\partial\Omega$.

Moreover, the term $\frac{C}{r}\displaystyle\iint_{\mathcal{K}}|u|^2\,dX$ appearing 
in \eqref{TTBBMM} may be replaced by the quantity
\begin{equation}\label{Eqqq-25}
Cr^{n-1}\left|\dint_{B_{\delta(A_r)/2}(A_r)}u(Z)\,dZ\right|^2,
\end{equation}
where $A_r$ is any point inside $\mathcal{K}$ (usually called a corkscrew point of $\Delta_r$).
\end{lemma}

\begin{proof} Let $\Delta_r=B_r(Q)\cap \partial\Omega$ be as in the statement of our Lemma. Writing $Q$ as $(q_0,q')$, let $\zeta$ be a smooth cutoff function of the form $\zeta(x_0, x')=\zeta_{0}(x_0)\zeta_{1}(x')$ where
\begin{equation}\label{Eqqq-27}
\zeta_{0}= 
\begin{cases}
1 & \text{ in } (-\infty, r_0+r], 
\\
0 & \text{ in } [r_0+2r, \infty),
\end{cases}
\qquad
\zeta_{1}= 
\begin{cases}
1 & \text{ in } B_{r}(q'), 
\\
0 & \text{ in } \mathbb{R}^{n}\setminus B_{2r}(q')
\end{cases}
\end{equation}
and
\begin{equation}\label{Eqqq-28}
r|\partial_{0}\zeta_{0}|+r|\nabla_{x'}\zeta_{1}|\leq c
\end{equation}
for some constant $c\in(0,\infty)$ independent of $r$. Here 
$r_0=6\sup_{x'\in B_r(q')}\hbar(x')$. Observe that our assumptions imply that
$$0\le r_0-6\hbar(x')\le  r_0 \lesssim r,\qquad \mbox{for all }x'\in B_{2r}(q').$$

Our goal is to control the $L^2$ norm of $u\big(\theta\hbar(\cdot),\cdot\big)$.  We fix $\alpha\in\{1,\dots,N\}$ and proceed to estimate
\begin{align}
&\hskip -0.20in
\int_{B_{r}(q')}u_{\alpha}^{2}(\theta\hbar(x'),x')\,dx' \le \int_{B_{2r}(q')}u_{\alpha}^{2}(\theta\hbar(x'),x')\zeta(\theta\hbar(x'),x')\,dx'
\nonumber\\[4pt]
&\hskip 0.70in
=-\iint_{\mathcal S(q',r,r_0,\theta\hbar)}\partial_{0}\left[u_{\alpha}^{2}(x_0,x')\zeta(x_0,x')\right]\,dx_0\,dx',
\nonumber
\end{align}
where $\mathcal S(q',r,r_0,\theta\hbar)=\{(x_0,x'):x'\in B_{2r}(q')\mbox{ and }\theta\hbar(x')<x_0<r_0+2r\}$. Hence:

\begin{align}\nonumber
&\hskip 0.10in
\mathcal I=\int_{B_{2r}(q')}u_{\alpha}^{2}(\theta\hbar(x'),x')\zeta(\theta\hbar(x'),x')\,dx' \le-2\iint_{\mathcal S(q',r,r_0,\theta\hbar)}u_{\alpha}\partial_{0}u_{\alpha}\zeta\,dx_0\,dx'  
\\[4pt]
&\hskip 0.70in
\quad-\iint_{\mathcal S(q',r,r_0,\theta\hbar)}u_{\alpha}^{2}(x_0,x')\partial_{0}\zeta\,dx_0\,dx'
=:\mathcal{A}+IV.\label{u6tg}
\end{align}
We further expand the term $\mathcal A$ as a sum of three terms obtained 
via integration by parts with respect to $x_0$ as follows:
\begin{align}\label{utAA}
\mathcal A &=-2\iint_{\mathcal S(q',r,r_0,\theta\hbar)}u_{\alpha}\partial_{0} 
u_{\alpha}(\partial_{0}x_0)\zeta\,dx_0\,dx' 
\nonumber\\[4pt]
&=2\iint_{\mathcal S(q',r,r_0,\theta\hbar)}\left|\partial_{0}u_{\alpha}\right|^{2}x_0\zeta\,dx_0\,dx' 
\nonumber\\[4pt]
&\quad +2\iint_{\mathcal S(q',r,r_0,\theta\hbar)}u_{\alpha}\partial^2_{00}u_{\alpha}x_0\zeta\,dx_0\,dx' 
\nonumber\\[4pt]
&\quad +2\iint_{\mathcal S(q',r,r_0,\theta\hbar)}u_{\alpha}\partial_{0}u_{\alpha}x_0\partial_{0}\zeta\,dx_0\,dx' 
\nonumber\\[4pt]
&=:I+II+III.
\end{align}

We start by analyzing the term $II$. As $\mathcal Lu=0$, \eqref{ES-new2ax} and the fact that $A_{00}^{\alpha\beta}=I_{N\times N}$ allows us to write
\begin{equation}\label{S3:T8:E01-x}
\partial^2_{00}u_{\alpha}
=-\left(\sum_{(i,j)\neq(0,0)}\partial_{i}\left({A}_{ij}^{\alpha\beta}\partial_{j}u_{\beta}\right)\right)-\partial_{0}\left({\varepsilon}_{0j}^{\alpha\beta}\partial_{j}u_{\beta}\right)
-B_i^{\alpha\beta}\partial_iu_\beta.
\end{equation}
Since $A_{0j}^{\alpha\beta}=0$ for $j>0$ this further simplifies to 
\begin{equation}\label{S3:T8:E01}
\partial^2_{00}u_{\alpha}
=-\left(\sum_{i>0}\partial_{i}\left({A}_{ij}^{\alpha\beta}\partial_{j}u_{\beta}\right)\right)-\partial_{0}\left({\varepsilon}_{0j}^{\alpha\beta}\partial_{j}u_{\beta}\right)
-B_i^{\alpha\beta}\partial_iu_\beta.
\end{equation}

In turn, this permits us to write the term $II$ as
\begin{align}
II &=-2\sum_{i>0}\iint_{\mathcal S(q',r,r_0,\theta\hbar)}u_{\alpha}\partial_{i}\left({A}_{ij}^{\alpha\beta}\partial_{j}u_{\beta}\right)x_0\zeta\,dx_0\,dx' 
\nonumber\\[4pt]
&\quad-2\iint_{\mathcal S(q',r,r_0,\theta\hbar)}u_{\alpha}\partial_{0}\left({\varepsilon}_{0j}^{\alpha\beta}\partial_{j}u_{\beta}\right)x_0\zeta\,dx_0\,dx' 
\nonumber\\[4pt]
&\quad-2\iint_{\mathcal S(q',r,r_0,\theta\hbar)}u_{\alpha}B_i^{\alpha\beta}\partial_iu_\beta x_0\zeta\,dx_0\,dx'
\nonumber\\[4pt]
&=2\sum_{i>0}\iint_{\mathcal S(q',r,r_0,\theta\hbar)}A_{ij}^{\alpha\beta}\partial_iu_{\alpha}\partial_{i}u_{\beta}x_0\zeta\,dx_0\,dx' 
\nonumber\\[4pt]
&\quad+2\sum_{i>0}\iint_{\mathcal S(q',r,r_0,\theta\hbar)}A_{ij}^{\alpha\beta}u_{\alpha}\partial_{i}u_{\beta}x_0\partial_i\zeta\,dx_0\,dx' 
\nonumber\\[4pt]
&\quad+2\iint_{\mathcal S(q',r,r_0,\theta\hbar)}\varepsilon_{0j}^{\alpha\beta}\partial_0u_{\alpha}\partial_{i}u_{\beta}x_0\zeta\,dx_0\,dx' 
\nonumber\\[4pt]
&\quad+2\iint_{\mathcal S(q',r,r_0,\theta\hbar)}\varepsilon_{0j}^{\alpha\beta}u_{\alpha}\partial_{i}u_{\beta}\zeta\,dx_0\,dx' 
\nonumber\\[4pt]
&\quad+2\iint_{\mathcal S(q',r,r_0,\theta\hbar)}\varepsilon_{0j}^{\alpha\beta}u_{\alpha}\partial_{i}u_{\beta}x_0\partial_0\zeta\,dx_0\,dx' 
\nonumber\\[4pt]
&\quad-2\iint_{\mathcal S(q',r,r_0,\theta\hbar)}B_i^{\alpha\beta}u_{\alpha}\partial_iu_\beta x_0\zeta\,dx_0\,dx'
\nonumber\\[4pt]
&\quad-2\sum_{i>0}\int_{\partial\mathcal S(q',r,r_0,\theta\hbar)}A_{ij}^{\alpha\beta}u_{\alpha}\partial_ju_\beta x_0\zeta\nu_i\,dS
\nonumber\\[4pt]
&=:II_{1}+II_{2}+II_{3}+II_{4}+II_{5}+II_{6}+II_{7}.\label{TFWW}
\end{align}
Here we have integrated by parts w.r.t. $\partial_i$. The boundary integral (term $II_7$) vanishes everywhere except on the graph of the function $\theta\hbar$ which implies that
\begin{align}
|II_7|&\le C\int_{B_{2r}(q')}|u_\alpha(\theta\hbar(\cdot),\cdot)\partial_ju_\beta (\theta\hbar(\cdot),\cdot) \theta\hbar(\cdot)\zeta(\theta\hbar(\cdot),\cdot)\nu_i|dS.\nonumber\\
&\le \frac12\int_{B_{2r}(q')}u_{\alpha}^{2}(\theta\hbar(x'),x')\zeta(\theta\hbar(x'),x')\,dx'\nonumber\\&\quad+C'
\int_{B_{2r}(q')}|\nabla u(\theta\hbar(x'),x')|^2|\hbar(x')|^2\,dx'=\frac12\mathcal I+II_8.
\end{align}
We can hide the first term as it is the same as on the lefthand side of \eqref{u6tg}, while the second term after integrating $II_8$ in $\theta$ becomes:
\begin{align}
\int_{1/6}^6|II_8|\,d\theta&\le C\int_{1/6}^6\int_{B_{2r}(q')}|\nabla u(\theta\hbar(x'),x')|^2 |\hbar(x')|^2dx'd\theta.\nonumber\\
&\lesssim\iint_{[0,r_0]\times B_{2r}(q')}|\nabla u|^2x_0\,dx_0\,dx'\lesssim\|S_b(u)\|^2_{L^2(B_{2r})}.
\end{align}

The remaining (solid integral) terms that are of the same type we estimate together. Firstly, we have 
\begin{equation}\label{Eqqq-29}
I+II_1+II_3\leq C(\lambda,\Lambda,n,N)\|S_b(u)\|^2_{L^2(B_{2r})}.
\end{equation}
Here, the estimate holds even if the square function truncated at a hight $O(r)$.
Next, since $r|\nabla\zeta|\le c$, if the derivative falls on the cutoff function $\zeta$ we have
\begin{align}\label{TDWW}
II_2+II_5 &\leq C(\lambda,\Lambda,n,N)\iint_{[0,2r]\times B_{2r}}\left|\nabla u\right||u|\frac{x_0}{r}\,dx_0\,dx' 
\nonumber\\[4pt]
&\leq C(\lambda,\Lambda,n,N)\left(\iint_{[0,2r]\times B_{2r}}|u|^{2}\frac{x_0}{r^{2}}\,dx_0\,dx'\right)^{1/2} 
\|S^{2r}_b(u)\|_{L^2(B_{2r})} 
\nonumber\\[4pt]
&\leq C(\lambda,\Lambda,n,N)\|S_b(u)\|_{L^2(B_{2r})}\|\tilde{N}_a(u)\|_{L^2(B_{2r})}.
\end{align}

The Carleson condition for \eqref{Car-m22} and the Cauchy-Schwarz inequality imply
\begin{equation}\nonumber
II_{4} + II_{6} \leq C(n,N)M^{1/2} 
\|S_b(u)\|_{L^2(B_{2r})}\|\tilde{N}_a(u)\|_{L^2(B_{2r})}.
\end{equation}
Finally, the interior term $IV$, which arises from the fact that $\partial_{0}\zeta$ vanishes on the set
$(-\infty,r_0+r)\cup(r_0+2r,\infty)$ may be estimated as follows:
\begin{equation}\label{Eqqq-31}
IV\leq\frac{C}{r}\iint_{[r_0+r,r_0+2r]\times B_{2r}(q')}|u|^{2}\,dx_0\,dx'.
\end{equation}
We put together all terms, sum them in $\alpha$ and integrate in $\theta$. The above analysis ultimately yields \eqref{TTBBMM}.
Finally, the last claim in the statement of the lemma that we can use \eqref{Eqqq-25} on the righthand side instead of the solid integral is a consequence of the Poicar\'e's inequality (see \cite{DHM} for detailed discussion).
\end{proof}

We now make use of Lemma~\ref{S3:L8-alt1}, involving the stopping time Lipschitz functions 
$\theta h_{\nu,a}(w)$, in order to obtain localized good-$\lambda$ inequality. We omit the proof as it is identical to the one given in \cite{DHM}. Here
$$Mf(x'):=\sup_{r>0}\dint_{|x'-z'|<r}|f(z')|\,dz'\mbox{ for }x'\in{\mathbb{R}}^{n-1},$$ 
denotes the standard Hardy-Littlewood maximal function on $\partial\BBR^n_+=\BBR^{n-1}$. 

\begin{lemma}\label{LGL-loc-alt1} Let $\mathcal L$ be an operator as in \eqref{ES-new2ax} with coefficients satisfying \eqref{Car-m22} and $A_{0j}^{\alpha\beta}(x)=\delta_{\alpha\beta}\delta_{0j}$.
Consider any boundary ball $\Delta_d=\Delta_d(Q)\subset {\mathbb R}^{n-1}$, let $A_d=(d/2,Q)$ be its corkscrew point and let
\begin{equation}
\nu_0=\left(\dint_{B_{d/4}(A_d)}|u(z)|^2\,dz\right)^{1/2}.
\end{equation}
Then for each $\gamma\in(0,1)$ there exists a constant $C(\gamma)>0$ 
such that $C(\gamma)\to 0$ as $\gamma\to 0$ and with the property that for each $\nu>2\nu_0$ and 
each energy solution $u$ of \eqref{E:D} there holds 
\begin{align}\label{eq:gl2}
&\hskip -0.20in 
\Big|\Big\{x'\in {\BBR}^{n-1}:\,\tilde{N}_a(u\chi_{T(\Delta_d)})>\nu,\,(M(S^2_b(u)))^{1/2}\leq\gamma\nu,
\nonumber\\[4pt] 
&\hskip 0in
\big(M(S^2_b(u))M(\tilde{N}_a^2(u\chi_{T(\Delta_d)}))\big)^{1/4}\leq\gamma\nu\Big\}\Big|
\nonumber\\[4pt] 
&\hskip 0.50in
\quad\le C(\gamma)\left|\big\{x'\in{\BBR}^{n-1}:\,\tilde{N}_a(u\chi_{T(\Delta_d)})(x')>\nu/32\big\}\right|.
\end{align}
Here $\chi_{T(\Delta_d)}$ is the indicator function of the Carleson region $T(\Delta_d)$ and the square function
$S_b$ in \eqref{eq:gl2} is truncated at the height $2d$. Similarly, the Hardy-Littlewood maximal operator $M$
is only considered over all balls $\Delta'\subset\Delta_{md}$ for some enlargement constant $m=m(a)\ge 2$.
\end{lemma}

Finally we have the following by the same argument as in \cite{DHM}.

\begin{proposition}\label{S3:C7-alt1} Under the assumptions of Lemma \ref{LGL-loc-alt1},
for any $p>0$ and $a>0$ there exists an integer $m=m(a)\ge 2$ and a finite 
constant $C=C(n,N,p,a,\|\mu\|_{\mathcal C})>0$ such that for all balls $\Delta_d\subset{\mathbb R}^{n-1}$ we have
\begin{equation}\label{S3:C7:E00ooloc}
\|\tilde{N}^r_a(u)\|_{L^{p}(\Delta_d)}\le C\|S^{2r}_a(u)\|_{L^{p}(\Delta_{md})}+Cd^{(n-1)/p}\left|\dint_{B_{\delta(A_d)/2}(A_d)}u(Z)\,dZ\right|,
\end{equation}
where $A_d$ denotes the corkscrew point of the ball $\Delta_d$.

We also have a global estimate for any $p>0$ and $a>0$. There exists a finite 
constant $C>0$ such that 
\begin{equation}\label{S3:C7:E00oo}
\|\tilde{N}_a(u)\|_{L^{p}({\BBR}^{n-1})}\le C\|S_a(u)\|_{L^{p}({\BBR}^{n-1})}.
\end{equation}
\end{proposition}

It remains to consider the system  \eqref{eq-sys-deriv}. The next Lemma is again an analogue of \cite[Lemma 5.4]{DHM}. 

\begin{lemma}\label{S3:L8-alt2} 
Let $\Omega=\BBR^n_+$ and let 
\begin{equation}\label{ES-old2aaa}
\mathcal{L}u=\left[ \partial_{i} \left({A}_{ij}^{\alpha \beta}(x) \partial_{j} u_{\beta}\right)\right]_\alpha
\end{equation}
has coefficients $A$ satisfying the Legendre-Hadamard condition \eqref{EllipLH}
and 
\begin{equation}\label{Car_hatAAxxxss}
d{\mu}(x)=\left(\sup_{B_{\delta(x)/2}(x)}|\nabla{A}|\right)^{2}
\delta(x)\,dx
\end{equation}
is a Carleson measure.

Then there exists $a>0$ with the following significance. Suppose $u$ is a weak solution of $\mathcal{L}u=0$ 
in $\Omega={\mathbb{R}}^n_{+}$. Select $\theta\in[1/6,6]$ and, having picked $\nu>0$ arbitrary,  
let $h_{\nu,a}(w)$ be as in \eqref{h}. Also, consider the domain 
$\mathcal{O}=\{(x_0,x')\in\Omega:\,x_0>\theta h_{\nu,a}(x')\}$ with boundary  
$\partial\mathcal{O}=\{(x_0,x')\in\Omega:\,x_0=\theta h_{\nu,a}(x')\}$. In this context, 
for any surface ball $\Delta_r=B_r(Q)\cap\partial\Omega$, with $Q\in\partial\Omega$ and $r>0$ 
chosen such that $h_{\nu,a}(w)\leq 2r$ pointwise on $\Delta_{2r}$, 
one has
\begin{align}\label{TTBBMM-alt2}
&\int_{\Delta_r}\big|\nabla u\big(\theta h_{\nu,a}(w)(\cdot),\cdot\big)\big|^2\,dx' 
\leq C(1+\|\mu\|^{1/2}_{\mathcal C})\|S_b(\nabla u)\|_{L^2(\Delta_{2r})}
\|\tilde{N}_a(\nabla u)\|_{L^2(\Delta_{2r})}
\nonumber\\
&\quad+\|\mu\|^{1/2}_{\mathcal C}\|\tilde{N}_a(\nabla u)\|_{L^2(\Delta_{2r})}^2+C\|S_b(\nabla u)\|^2_{L^2(\Delta_{2r})}+\frac{c}{r}\iint_{\mathcal{K}}|\nabla u|^{2}\,dX.
\end{align}
Here $C=C(\lambda,\Lambda,n,N)\in(0,\infty)$ and $\mathcal{K}$ is a region inside $\mathcal{O}$ of diameter, 
distance to the boundary $\partial\mathcal{O}$, and distance to $Q$, are all comparable to $r$. 
Again, the term $\frac{c}{r}\displaystyle\iint_{\mathcal{K}}|\nabla u|^2\,dX$ appearing 
in \eqref{TTBBMM} may be replaced by the quantity
\begin{equation}\label{Eqqq-25-alt2}
Cr^{n-1}\left|\dint_{B_{\delta(A_r)/2}(A_r)}\nabla u(Z)\,dZ\right|^2.
\end{equation}
\end{lemma}

\begin{proof} Consider the pullback of the operator $\mathcal L$ given by \eqref{ES-old2aaa} from $\mathcal O$
onto $\BBR^n_+$ as in Lemma 5.4 of \cite{DHM}. Slightly abusing the notation we still call this operator $\mathcal L$ with coefficients $A$. We repeat the calculation \eqref{u6tg} with
$w^k=\partial_k u$ in place of $v$. Let us focus on the term analogous to $II$ in \eqref{utAA}. When $k>0$
we have
\begin{align}\label{xxdggr}
II&=\quad2\iint_{[0,2r]\times B_{2r}(y')}w^k_{\alpha}\partial^2_{00}w^k_{\alpha}x_0\zeta\,dx_0\,dx' \\\nonumber
&=\quad2\iint_{[0,2r]\times B_{2r}(y')}w^k_{\alpha}\partial_{k}\partial_0w^0_{\alpha}x_0\zeta\,dx_0\,dx'\\\nonumber
&=-2\iint_{[0,2r]\times B_{2r}(y')}\partial_kw^k_{\alpha}\partial_0w^0_{\alpha}x_0\zeta\,dx_0\,dx''\\\nonumber
&\quad-2\iint_{[0,2r]\times B_{2r}(y')}w^k_{\alpha}\partial_0w^0_{\alpha}x_0\partial_k\zeta\,dx_0\,dx'=II_1+II_2.
\end{align}
Hence clearly, 
$$|II_1|\le C\|S_b(\nabla u)\|^2_{L^2(\Delta_{2r})},$$
and
$II_2$ is a term analogous to \eqref{TDWW} with the corresponding estimate 
$$|II_2|\le C\|S_b(\nabla u)\|_{L^2(\Delta_{2r})}
\|\tilde{N}_a(\nabla u)\|_{L^2(\Delta_{2r})}.$$
The remaining terms are analogous to \eqref{u6tg} and we omit them for the sake of brevity.

We still have to get an estimate for $w^0=\partial_0u$.
Instead, it is more convenient to work with
$$H_\alpha=\sum_{j,\beta}{A}_{0j}^{
\alpha\beta}{w}^j_\beta.$$
Recall that if a linear transformation $T:{\mathbb R}^N\to {\mathbb R}^N$ is invertible, then for some positive constants $c_1,c_2$ we have $c_1|x|\le |Tx|\le c_2|x|$. Using this fact for $x\mapsto A_{00}^{\alpha\beta}x_\beta$ (the matrix $A_{00}^{\alpha\beta}$ is invertible since we assume \eqref{EllipLH}) we have
\begin{align}\label{u6tg-otoh}
&
\int_{B_{2r}(y')}|w^0(0,x')|^{2}\zeta(0,x')\,dx' \approx \sum_{\alpha}\int_{B_{2r}(y')}|{A}_{00}^{\alpha\beta}w^0_\beta(0,x')|^{2}\zeta(0,x')\,dx' 
\nonumber\\
&
\le C(n,N)\sum_{\alpha}\left[\int_{B_{2r}(y')}(|H_\alpha)|^{2}\zeta)(0,x')\,dx'+\sum_{j>0}\int_{B_{2r}(y')}(|A^{\alpha\beta}_{0j}w^j_\beta|^{2}\zeta)(0,x')\,dx'  \right]
\nonumber\\
&
\le C(n,N)\int_{B_{2r}(y')}(|H|^{2}\zeta)(0,x')\,dx'+C(n,N,\Lambda) \int_{B_{r}(y')}|\nabla_T{u}(0,x')|^2\,dx' .
\end{align}
The second term is OK as we have just verified the estimate for the tangential directions in \eqref{xxdggr}. We deal with the first term now. A calculation similar to 
\eqref{u6tg}-\eqref{utAA} gives us
\begin{align}\label{u6tg-x}
&\hskip -0.20in
\int_{B_{2r}(y')}|H|^{2}(0,x')\zeta(0,x')\,dx' 
\nonumber\\[4pt]
&\hskip 0.70in
=-2\iint_{[0,2r]\times B_{2r}(y')} H_\alpha\,\partial_0 H_\alpha\zeta\,dx_0\,dx' 
\nonumber\\[4pt]
&\hskip 0.70in
\quad-\iint_{[0,2r]\times B_{2r}(y')}|H|^{2}(x_0,x')\partial_{0}\zeta\,dx_0\,dx'.
\end{align}
The second term has a similar estimate as \eqref{Eqqq-31}. For the first term we use the fact that ${{\mathcal L}{u}}=0$ which implies that
$$\partial_0 H_\alpha=-\sum_{i>0}\partial_i({A}^{\alpha\beta}_{ij}w^j_\beta).$$
It follows
\begin{align}\label{u6tg-xx}
&\hskip -0.20in
-2\iint_{[0,2r]\times B_{2r}(y')}H_\alpha\,\partial_0 H_\alpha\zeta\,dx_0\,dx' 
\nonumber\\[4pt]
&\hskip 0.20in
=\quad2\sum_{i>0}\iint_{[0,2r]\times B_{2r}(y')}H_\alpha\,\partial_i(A_{ij}^{\alpha\beta}w^j_\beta)\zeta(\partial_0x_0)\,dx_0\,dx' \nonumber
\end{align}
\begin{align}
&\hskip 0.20in
=-2\sum_{i>0}\iint_{[0,2r]\times B_{2r}(y')}\partial_0H_\alpha\,\partial_i(A_{ij}^{\alpha\beta}w^j_\beta)\zeta\,x_0\,dx_0\,dx'
\nonumber\\[4pt]
&\hskip 0.20in
\quad+2\sum_{i>0}\iint_{[0,2r]\times B_{2r}(y')}\partial_iH_\alpha\,\partial_0(A_{ij}^{\alpha\beta}w^j_\beta)\zeta\,x_0\,dx_0\,dx'
\nonumber\\[4pt]
&\hskip 0.20in
\quad-2\sum_{i>0}\iint_{[0,2r]\times B_{2r}(y')}H_\alpha\,\partial_i(A_{ij}^{\alpha\beta}w^j_\beta)(\partial_0\zeta)\,x_0\,dx_0\,dx'
\nonumber\\[4pt]
&\hskip 0.20in
\quad+2\sum_{i>0}\iint_{[0,2r]\times B_{2r}(y')}H_\alpha\,\partial_0(A_{ij}^{\alpha\beta}w^j_\beta)(\partial_i\zeta)\,x_0\,dx_0\,dx'.
\end{align}

We analyze this term by term. In the last two terms, if the derivative falls on $w^j_\beta$ these terms are of the same nature as the term $II_2$ and is handled identically. When the derivative falls on coefficients these are bounded by
$$\iint_{[0,2r]\times B_{2r}(y')}|w|^2|\nabla {A}|\frac{x_0}r\,dx_0\,dx'\lesssim \|\mu\|^{1/2}_{\mathcal C}\|N_a(\nabla u)\|^2_{L^2},$$
where we have used the Cauchy-Schwarz inequality and the Carleson condition. 

The first two terms on the righthand side of \eqref{u6tg-xx} will give us the square function of $w=\nabla u$ when both derivatives fall on $w$ or a mixed term like $II_2$ above or finally when both derivatives hit the coefficients we get terms bounded from above by
$$\iint_{[0,2r]\times B_{2r}(y')}|w|^2|\nabla {A}|^2x_0\,dx_0\,dx'\lesssim \|\mu\|_{\mathcal C}\|N_a(\nabla u)\|^2_{L^2}.$$
With this in hand, the estimate in \eqref{TTBBMM-alt2} follows.
\end{proof}

From this we have as before:

\begin{lemma}\label{LGL-loc-alt2} 
Consider the system ${\mathcal L}u=0$, where $\mathcal L$ is as in Lemma~\ref{S3:L8-alt2} 
with coefficients satisfying the Carleson condition \eqref{Car_hatAAxxxss} and the ellipticity condition \eqref{EllipLH}.

Consider any boundary ball $\Delta_d=\Delta_d(Q)\subset {\mathbb R}^{n-1}$, let $A_d=(d/2,Q)$ be its corkscrew point and let
\begin{equation}
\nu_0=\left(\dint_{B_{d/4}(A_d)}|\nabla u(z)|^2\,dz\right)^{1/2}.
\end{equation}
Then for each $\gamma\in(0,1)$ there exists a constant $C(\gamma)>0$ 
such that $C(\gamma)\to 0$ as $\gamma\to 0$ and with the property that for each $\nu>2\nu_0$ and 
each energy solution $u$ of \eqref{E:D} there holds 
\begin{align}\label{eq:gl2-alt2}
&\hskip -0.20in 
\Big|\Big\{x'\in {\BBR}^{n-1}:\,\tilde{N}_a(\nabla u\chi_{T(\Delta_d)})>\nu,\,(M(S^2_b(\nabla u)))^{1/2}\leq\gamma\nu,
\nonumber\\[4pt] 
&\hskip 0in
\big(M(S^2_b(\nabla u))M(\tilde{N}_a^2(\nabla u\chi_{T(\Delta_d)}))\big)^{1/4}\leq\gamma\nu,\,\big(M(\|\mu\|^{1/2}_{\mathcal C}\tilde{N}_a^2(\nabla u\chi_{T(\Delta_d)}))\big)^{1/2}\leq\gamma\nu\Big\}\Big|
\nonumber\\[4pt] 
&\hskip 0.50in
\quad\le C(\gamma)\left|\big\{x'\in{\BBR}^{n-1}:\,\tilde{N}_a(\nabla u\chi_{T(\Delta_d)})(x')>\nu/32\big\}\right|.
\end{align}
Here $\chi_{T(\Delta_d)}$ is the indicator function of the Carleson region $T(\Delta_d)$ and the square function
$S_b$ in \eqref{eq:gl2} is truncated at the height $2d$. Similarly, the Hardy-Littlewood maximal operator $M$
is only considered over all balls $\Delta'\subset\Delta_{md}$ for some enlargement constant $m=m(a)\ge 2$.
\end{lemma}

From this:

\begin{proposition}\label{S3:C7-alt2} Under the assumptions of Lemma \ref{LGL-loc-alt2}, for any $p>0$ and $a>0$ there exists an integer $m=m(a)\ge 2$ and a finite 
constants $K=K(n,N,\lambda,\Lambda,p,a)>0$ and $C=C(n,N,\lambda,\Lambda,p,a)>0$ such that if
$$\|\mu\|_{\mathcal C}< K,$$
then for all balls $\Delta_d\subset{\mathbb R}^{n-1}$ we have
\begin{equation}\label{S3:C7:E00ooloc-alt2}
\|\tilde{N}^r_a(\nabla u)\|_{L^{p}(\Delta_d)}\le C\|S^{2r}_a(\nabla u)\|_{L^{p}(\Delta_{md})}+Cd^{(n-1)/p}\left|\dint_{B_{\delta(A_d)/2}(A_d)}\nabla u(Z)\,dZ\right|,
\end{equation}
where $A_d$ denotes the corkscrew point of the ball $\Delta_d$.

We also have the global estimate:
\begin{equation}\label{S3:C7:E00oo-alt2}
\|\tilde{N}_a(\nabla u)\|_{L^{p}({\BBR}^{n-1})}\le C\|S_a(\nabla u)\|_{L^{p}({\BBR}^{n-1})}.
\end{equation}
\end{proposition}
The condition $\|\mu\|_{\mathcal C}$ being small in this Proposition is needed due to the presence of the term $\big(M(\|\mu\|^{1/2}_{\mathcal C}\tilde{N}_a^2(\nabla u\chi_{T(\Delta_d)}))\big)^{1/2}\leq\gamma\nu$ in the good-$\lambda$
inequality (Lemma \ref{LGL-loc-alt2}). This term can be absorbed by the lefthand side of \eqref{S3:C7:E00ooloc-alt2} when $\mu$ has small Carleson norm.

\section{Proof of Theorem \ref{S3:T0}}\label{S6}

We start with $p=2$.
Let ${\mathcal L}_1$
be an operator on domain $\Omega$ whose coefficients satisfy assumptions of Theorem  \ref{S3:T0}. Denote these coefficients by $A$, $B$ so that we have that
\begin{equation}\label{oscT1a}
d\mu_1 =  \left( \delta(x)^{-1}\left(\mbox{osc}_{B_{\delta(x)/2}(x)}A\right)^2 + \sup_{B_{\delta(x)/2}(x)} |B|^2 \delta(x)\right) dx
\end{equation}
is a Carleson measure with norm $\|\mu_1\|_{\mathcal C}$ and $A$ is strongly elliptic.

As discussed in section 2.4 of \cite{DHM} the problem can be reduced to the case $\Omega=\mathbb R^n_+$
(via a pullback map $\rho:\Omega\to \mathbb R^n_+$). Slightly abusing the notation we still call it ${\mathcal L}_1$ and its coefficients are still the tensors $A$ and $B$. The pullback preserves strong ellipticity but might disrupt the assumption $(i)$ that ${A}_{0j}^{\alpha\beta}=\delta_{\alpha\beta}\delta_{0j}$. However, if the Lipschitz norm $L$ of the map $\phi$ whose graph defines the boundary of $\partial\Omega$ is small we get that the new coefficients satisfy
\begin{equation}\label{eqS6.1}
|{A}_{0j}^{\alpha\beta}-\delta_{\alpha\beta}\delta_{0j}|<\varepsilon
\end{equation}
for some small $\varepsilon$ (depending on $L$).

Consider a new operator ${\mathcal L}$ whose coefficients are
defined as follows. 
Set
$$\bar{A}_{ij}^{\alpha\beta}(x_0,x')= \iint_{{\mathbb R}^n_+} A_{ij}^{\alpha\beta}(s,u)\varphi_{t}(s-t,x'-u)dsdu,$$

where $\varphi$ is a smooth real, nonnegative bump function on ${\mathbb R}^n$ supported in
the ball  $B_{1/2}(0)$ such that $\iint \varphi=1$ and
$\varphi_t(s,y) = t^{-n}\varphi(s/t,y/t)$. We also set $\bar B=B$.

The Carleson norms of
\begin{equation}\label{Car_hatAA-X}
d{\mu}(x)=\sup_{B_{\delta(x)/2}(x)}\left[|\nabla{\bar A}|^{2} + |\bar B|^{2} \right]\delta(x)\,dx
\end{equation}
and 
\begin{equation}\label{eqdm-X}
 dm(x)=\sup_{B_{\delta(x)/2}(x)}\left[|\bar A - A|^{2} \delta^{-1}(x)+ |\bar B - B|^{2} \delta(x)\right]\,dx
\end{equation}
satisfy, by the same arguments as in \cite[Corollary 2.3]{DPP}, the following bounds:
$$\|\mu\|_{\mathcal C}+\|m\|_{\mathcal C}\le C\|\mu_1\|_{\mathcal C},$$
for some $C=C(n,N,\phi)\ge 1$.

Because \eqref{eqS6.1} holds the same is true for the coefficients $\bar A^{\alpha\beta}_{0j}$. Hence we can apply the transformation of section 2.3 of \cite{DHM} to get a new operator the satisfies $(i)$ of Theorem \ref{S3:T0}.
Also since $\varepsilon$ in \eqref{eqS6.1} is small, (as $L$ is) this transformation preserves strong ellipticity of the matrix. Hence without loss of generality we may assume that 
$$\bar{A}_{0j}^{\alpha\beta}=\delta_{\alpha\beta}\delta_{0j}.$$
It follows we are in the situation we can apply Lemma \ref{S3:L4-alt1}, Corollary \ref{S4:C1-alt1} and Proposition \ref{S3:C7-alt1}.
It follows by these three results that for any energy solution $u$ of ${\mathcal L}_1=0$ we have
\begin{align}
&\int_{\BBR^{n-1}}\left[\tilde{N}_{a}(u)\right]^{p}\,dx'\le C_1\int_{\BBR^{n-1}}\left[{S}_{a}(u)\right]^{p}\,dx'\\\nonumber
&\le C_1C_2\left(\int_{\BBR^{n-1}}|u(0,x')|^{2}\,dx'
+(\|\mu\|_{\mathcal{C}}+\|m\|_{\mathcal{C}})\int_{\BBR^{n-1}}\left[\tilde{N}_a(u)\right]^{2}\,dx'\right)\\\nonumber
&\le C_1C_2C_3\left(\int_{\BBR^{n-1}}|u(0,x')|^{2}\,dx'
+\|\mu_1\|_{\mathcal{C}}\int_{\BBR^{n-1}}\left[\tilde{N}_a(u)\right]^{2}\,dx'\right)
\end{align}
From this if $C_1C_2C_3\|\mu_1\|_{\mathcal{C}}<1/2$ we obtain \eqref{Main-Est} which gives us $L^2$ solvability of the Dirichlet problem for the operator ${\mathcal L}_1$.\vglue1mm

Once we have the result for $p=2$ then again as in \cite{DHM} we obtain solvability in the range 
$2-\varepsilon<p<\frac{2(n-1)}{(n-2)}+\varepsilon$ by extrapolation. We refer the reader to the section 6 of \cite{DHM} where this is described in detail. For the interval $p\in (2,\frac{2(n-1)}{(n-2)}+\varepsilon)$ the proof uses Theorem 1.2 of \cite{S4} while the interval $p\in (2-\varepsilon,2)$ is handled by a real variable argument.

\section{Proof of Theorem \ref{S3:T3}}

The proof of Theorem \ref{S3:T3} is fairly straightforward and is similar to the argument given in the previous section. An operator ${\mathcal L}$ as in \eqref{ES-4} on domain $\Omega$ is rewritten as
\begin{equation}\label{ES-2-fd}
\mathcal{L}u=\left[ \partial_{i} \left(A_{ij}^{\alpha \beta}(x) \partial_{j} u_{\beta}\right)
+B_{i}^{\alpha \beta}(x) \partial_{i}u_{\beta}\right]_{\alpha}
\end{equation}
with $A_{0j}^{\alpha\beta}=\delta_{\alpha\beta}\delta_{0j}$ and $A$, $B$ satisfying the small Carleson condition and strong ellipticity. Then by the same argument as in the previous section the matters can be reduced to $\Omega={\mathbb R}^n_+$. 

It follows we are in the situation where we can apply Lemma \ref{S3:L4-alt2} and obtain 
$$\int_{\BBR^{n-1}}\left[{S}_{a}(\nabla u)\right]^{p}\,dx'
\le C_1\left(\int_{\BBR^{n-1}}|\nabla_T u(0,x')|^{2}\,dx'
+\|\mu\|_{\mathcal{C}}\int_{\BBR^{n-1}}\left[\tilde{N}_a(\nabla u)\right]^{2}\,dx'\right).$$

Next we want to apply Proposition \ref{S3:C7-alt2} but this proposition was only established for operators without first order terms. We therefore go to our original operator \eqref{ES-4} on domain $\Omega$ and use the pullback map $\rho$ mentioned earlier (c.f. section 2.4 of \cite{DHM}). This gives us a new operator on ${\mathbb R}^n_+$ still of the form \eqref{ES-4} and hence Proposition \ref{S3:C7-alt2} applies to it. This gives
$$\int_{\BBR^{n-1}}\left[\tilde{N}_{a}(\nabla u)\right]^{p}\,dx'\le C_2\int_{\BBR^{n-1}}\left[{S}_{a}(\nabla u)\right]^{p}\,dx'.$$
Combining these we finally have
\begin{align}\nonumber
&\int_{\BBR^{n-1}}\left[\tilde{N}_{a}(\nabla u)\right]^{p}\,dx'\le C_1C_2\left(\int_{\BBR^{n-1}}|\nabla_T u(0,x')|^{2}\,dx'
+\|\mu\|_{\mathcal{C}}\int_{\BBR^{n-1}}\left[\tilde{N}_a(\nabla u)\right]^{2}\,dx'\right).
\end{align}
From this if $C_1C_2\|\mu_1\|_{\mathcal{C}}<1/2$ we obtain \eqref{Main-Estxx} which implies $L^2$ solvability of the Regularity problem for the operator ${\mathcal L}$.\vglue1mm

We may also establish a local version of the calculation given above.
Consider domains of the following the form. Let $\Delta_d\subset{\mathbb R}^{n-1}$
be a boundary ball or a cube or diameter $d$. We denote by ${\mathcal O}_{\Delta_d,a}$ 
\begin{equation}\label{Odom}
{\mathcal O}_{\Delta_d,a}=\bigcup_{Q\in\Delta_d}\Gamma_a(Q).
\end{equation}

\begin{lemma}\label{S7:L1} Let $\mathcal L$ be as in Theorem \ref{S3:T3} on the domain ${\mathbb R}^n_{+}$. There exists large $a>0$ with the following property.
If $\mathcal O$ is a Lipschitz domain defined by \eqref{Odom} and $u$ is any energy solution of $\mathcal L u=0$ with the Dirichlet boundary datum $\nabla_T f\in L^{2}(\partial \mathcal O;{\BBR}^N)$ then the following estimate holds:
\begin{equation}\label{Main-Estlocxx2bb}
\|\tilde{N}_{a/2} (\nabla u)\|_{L^{2}(\Delta_d)}\leq C\|\nabla_T f\|_{L^{2}(\partial{\mathcal O\cap \overline{T(\Delta_{md})}};{\BBR}^N)}+Cd^{(n-1)/2}\sup_{x\in \mathcal O\cap\{\delta(x)>d\}}W(x),
\end{equation}
where $\delta(x)=\mbox{dist}(x,\partial{\mathbb R}^n_+)$, $W(x)=\left(\dint_{B_{\delta(x)/4}(x)}|\nabla u(y)|^2 dy)\right)^{1/2}$ and $m=m(a)>1$ is sufficiently large.
\end{lemma}

\begin{proof} The lemma can be established by combining the local results 
\eqref{eq5.15-alt2} and \eqref{S3:C7:E00ooloc-alt2}. 
In last term of \eqref{Main-Estlocxx2bb} because of the way $\mathcal O$ is defined we clearly have
\begin{equation}\label{brmbrm}
\{(x_0,x')\in\mathcal O:\, x'\notin\Delta_{(1+a)d}\}\subset  \mathcal O\cap\{\delta(x)>d\}.
\end{equation}
It follows that again by considering the pullback map $\rho:{\mathbb R}^n_+\to\mathcal O$ of section 2.4 of \cite{DHM} proving \eqref{Main-Estlocxx2bb} is equivalent to establishing
\begin{equation}\label{Main-Estlocxxy}
\|\tilde{N} (\nabla u)\|_{L^{2}(\Delta_d)}\leq C\|\nabla_Tf\|_{L^{2}(\Delta_{md};{\BBR}^N)}+Cd^{(n-1)/2}\sup_{x\in {\mathbb R}^n_+\setminus T(\Delta_{(1+a)d})}W(x),
\end{equation}
where we now work on the domain ${\mathbb R}^n_+$ with $u$ solving $\mathcal Lu=0$ in ${\mathbb R}^n_+$ for $\mathcal L$ as in Theorem \ref{S3:T3}. We start with the term on the lefthand side of \eqref{Main-Estlocxxy}. It follows from \eqref{S3:C7:E00ooloc-alt2} that
\begin{equation}\label{S3:C7:E00ooloc-2}
\|\tilde{N}^{(1+a)d}(\nabla u)\|^2_{L^{2}(\Delta_d)}\le C\int_{T(\Delta_{md})}|\nabla^2 u|^2\delta(x)\,dx +Cd^{n-1}|(\nabla u)_{av}(A_d)|^2.
\end{equation}
The last term above has a trivial bound by $Cd^{n-1}\sup_{x\in {\mathbb R}^n_+\setminus T(\Delta_{(1+a)d})}[W(x)]^2$. To estimate the first term on the righthand side of \eqref{S3:C7:E00ooloc-2} we use \eqref{eq5.15-alt2}. This gives
\begin{eqnarray}\label{altsim}
&&\int_{T(\Delta_{md})}|\nabla^2 u|^2\delta(x)\,dx\\
&\lesssim& 
\int_{\Delta_{4md}}|\nabla_Tu(0,x')|^{2}\,dx' 
+\int_{\Delta_{4md}}|\nabla_Tu(2md,x')|^{2}\,dx'\nonumber\\
&+&\|\mu\|_{\mathcal{C}}\int_{\Delta_{4md}}\left[\tilde{N}^{2md}(\nabla u)\right]^{2}\,dx'.\nonumber
\end{eqnarray}
The second term in the last line can be estimated by $Cd^{n-1}\sup_{x\in {\mathbb R}^n_+\setminus T(\Delta_{(1+a)d})}[W(x)]^2$
using an averaging procedure. By varying $d$ in \eqref{altsim} between say $d_0$ to $2d_0$ the second term
turns into a solid integral over a set that is contained in ${\mathbb R}^n_+\setminus T(\Delta_{(1+a)d})$ and hence the estimate holds. This gives
\begin{eqnarray}\label{altsim2}
&&\int_{T(\Delta_{md})}|\nabla^2 u|^2\delta(x)\,dx\\
&\lesssim& 
\int_{\Delta_{8md}}|\nabla_Tf(x')|^{2}\,dx' 
+\|\mu\|_{\mathcal{C}}\int_{\Delta_{8md}}\left[\tilde{N}^{4md}(\nabla u)\right]^{2}\,dx'\nonumber\\ &+&d^{n-1}\sup_{x\in {\mathbb R}^n_+\setminus T(\Delta_{(1+a)d})}[W(x)]^2.\nonumber
\end{eqnarray}
Finally for the second term in the middle line we again use  \eqref{S3:C7:E00ooloc-alt2}. We get
\begin{eqnarray}\label{altsim3}
&&\int_{T(\Delta_{md})}|\nabla^2 u|^2\delta(x)\,dx\\
&\lesssim& 
\int_{\Delta_{8md}}|\nabla_T f(x')|^{2}\,dx' 
+\|\mu\|_{\mathcal{C}}\int_{T(\Delta_{8m^2d})}|\nabla^2 u|^2\delta(x)\,dx\nonumber\\
&+&d^{n-1}\sup_{x\in {\mathbb R}^n_+\setminus T(\Delta_{(1+a)d})}[W(x)]^2.\nonumber
\end{eqnarray}
For sufficiently small $\|\mu\|_{\mathcal{C}}$ we can hide part of the second term in the last line on the righthand side of \eqref{altsim3}. Hence
\begin{eqnarray}\label{altsim4}
&&\int_{T(\Delta_{md})}|\nabla^2 u|^2\delta(x)\,dx \lesssim 
\int_{\Delta_{8md}}|\nabla_Tf(x')|^{2}\,dx' \\
&+&\|\mu\|_{\mathcal{C}}\int_{T(\Delta_{8m^2d})\setminus T(\Delta_{md})}|\nabla^2 u|^2\delta(x)\,dx
+d^{n-1}\sup_{x\in {\mathbb R}^n_+\setminus T(\Delta_{(1+a)d}))}[W(x)]^2.\nonumber
\end{eqnarray}
We claim that by the Caccioppoli inequality for the second gradient (Proposition \ref{Ciac2}) we have
\begin{equation}\label{cacc}
\int_{T(\Delta_{8m^2d})\setminus T(\Delta_{md})}|\nabla^2 u|^2\delta(x)\,dx
\lesssim d^{n-1}\sup_{x\in {\mathbb R}^n_+\setminus T(\Delta_{(1+a)d}))}[W(x)]^2.
\end{equation}
This is obvious on the set $T(\Delta_{8m^2d})\cap\{\delta(x)\ge d\}$ which is clearly in the interior of ${\mathbb R}^n_+$. However, let us recall \eqref{brmbrm}. It follows that all points of $T(\Delta_{8m^2d})\setminus T(\Delta_{md})$ are in the interior of the original domain ${\mathcal O}$ and hence we can use Caccioppoli inequality in the original domain.

Finally, by combining \eqref{S3:C7:E00ooloc-2}, \eqref{altsim4} and \eqref{cacc} we see that \eqref{Main-Estlocxxy} holds. We can remove the truncation of $\tilde N$ at height $(1+a)d$ in \eqref{S3:C7:E00ooloc-2} as for points above this height the term $d^{n-1}\sup_{x\in {\mathbb R}^n_+\setminus T(\Delta_{(1+a)d}))}[W(x)]^2$ controls the nontangential maximal function.
\end{proof}

\vglue1mm

Next, we set
$$E_{\nu}=\{x'\in{\mathbb R}^{n-1}:\,\tilde{N_\alpha}(\nabla u)(x')>\nu \}.$$
Here, $\alpha>0$ (much larger than $a$) is determined later. With $f=u\big|_{\partial{\mathbb R}^n_+}$
we denote by $g$
$$g(x')=\sup_{B\ni x'}\left(\dint_{B}|\nabla_T f|^2(y')dy'\right)^{1/2},$$
for all $x'\in{\mathbb R}^{n-1}$ where the supremum is taken over all boundary balls $B$ containing $x$.

We now prove the following estimate which is an analogue of (2.15) from \cite{DK}.
\begin{equation}\label{Main-Estloc5zz}
\int_{E_\nu\cap\{g\le\nu\}}\left[\tilde{N} (\nabla u)(x')\right]^2\,dx'\le C\nu^2|E_\nu| +C\alpha^{-1}\int_{E_\nu}\left[\tilde{N} (\nabla )(x')\right]^2\,dx'.
\end{equation}

Let $(\Delta^i)$ be the Whitney decomposition of $E_\nu$ with the property that $2m\Delta^i\subset E_\nu$ and
$2m\Delta^i$ have finite overlaps. Here $m$ is chosen as in Lemma \ref{S7:L1}.
We look at those Whitney cubes such that
$$F^i=\Delta^i\cap \{x':\, g(x')\le \nu\}\ne\emptyset.$$
Since $\Delta^i$ is the Whitney cube there exists a point $x_i\in{\mathbb R}^{n-1}\setminus E_\nu$ with 
$$\mbox{dist}(x_i,\Delta^i)\le C_n\mbox{ diam}(\Delta^i).$$

For $1<\tau<2$ consider the Lipschitz domains
$$\Omega_\tau={\mathcal O}_{\tau\Delta^i,a}$$ 
where $\tau\Delta^i$ is an enlargement of $\Delta^i$ by factor of $\tau$ and $a$ was chosen earlier (so that the solvability of $\Omega_\tau$ holds). Set $A_\tau=\partial\Omega_\tau\cap \Gamma_\alpha(x_i)$, $B_\tau=(\partial\Omega_\tau\cap {\mathbb R}^n_+)\setminus \Gamma_\alpha(x_i)$.

Because of the choices we have made for $\tau\in (1,2)$ the height of $B_\tau$ is bounded, namely we have
\begin{equation}\label{hhh}
h:=\sup\{y_0:\, (y_0,y')\in B_\tau\}\le C_n \alpha^{-1} \mbox{ diam}(\Delta^i).
\end{equation}
Since $F^i\ne\emptyset$ we have
\begin{equation}\label{xnfrk}
\int_{2m\Delta^i}|\nabla_Tf(x')|^2dx'\lesssim \int_{2m\Delta^i}|g(x')|^2dx'\lesssim \nu^2|\Delta^i|.
\end{equation}
It follows by Lemma \ref{S7:L1} for each $\Omega_\tau$ we have by \eqref{Main-Estlocxxy}
\begin{equation}\label{Main-Estloc2}
\|\tilde{N} (\nabla u)\|^2_{L^{2}(\Delta^i)}\leq C\|\nabla_Tf\|^2_{L^{2}(\partial\Omega_\tau\cap \overline{T(2m\Delta^i)};{\BBR}^N)}+Cd^{n-1}\sup_{x\in {\Omega}_\tau\cap\{\delta(x)>d\}}[W(x)]^2.
\end{equation}
Here $d=\diam(\Delta^i)$ and $\tilde{N}$ is defined using cones $\Gamma_b$ (see above).
We deal with the terms on the righthand side. Firstly, for sufficiently large $\alpha>0$ we have ${\Omega}_\tau\cap\{\delta(x)>d\}\subset\Gamma_\alpha(x_i)$ and hence 
$$d^{n-1}\sup_{x\in {\Omega}_\tau\cap\{\delta(x)>d\}}[W(x)]^2\lesssim \nu^2|\Delta^i|.$$
The boundary $\partial\Omega_\tau$ consists of three pieces, $A_\tau$, $B_\tau$ and $\partial\Omega_\tau\cap {\mathbb R}^{n-1}\subset 2m\Delta^i$, for the last piece we already have the estimate \eqref{xnfrk}. Hence by 
\eqref{Main-Estloc2}
\begin{equation}\label{Main-Estloc3}
\|\tilde{N} (\nabla u)\|^2_{L^{2}( \Delta^i)}\le C\|\nabla u\|^2_{L^{2}(A_\tau\cap T(2m\Delta^i))}+C\|\nabla u\|^2_{L^{2}(B_\tau)}+C\nu^2|\Delta^i|.
\end{equation}
We integrate \eqref{Main-Estloc3} in $\tau$ over the interval $(1,2)$ in $\tau$. Since $A_\tau\subset \Gamma_\alpha(x_i)$ integrating
in $\tau$ turns this into a solid integral which has the following estimate
$$\int_1^2\|\nabla u\|^2_{L^{2}(A_\tau)}d\tau \lesssim d^{-1}\iint_{\bigcup_\tau A_\tau}|\nabla u(x)|^2dx\lesssim
d^{-1}\iint_{\Gamma_\alpha(x_i)\cap T(2m\Delta^i)}|\nabla u(x)|^2dx,$$
with the last term bounded by $C \nu^2|\Delta^i|$. We have a similar estimate for $B_\tau$. 
$$\int_1^2\|\nabla u\|^2_{L^{2}(B_\tau)}d\tau \lesssim d^{-1}\iint_{\bigcup_\tau B_\tau}|\nabla u(x)|^2dx\lesssim
d^{-1}\iint_{T(2m\Delta^i)\cap \{x\le h\}}|\nabla u(x)|^2dx.$$
However thanks to \eqref{hhh} we conclude
$$d^{-1}\iint_{T(2m\Delta^i)\cap \{x\le h\}}|\nabla u(x)|^2dx\lesssim d^{-1}\alpha^{-1}d\int_{2m{\Delta^i}}|\tilde{N}(\nabla u)(x')|^2dx'.$$

Putting all terms together yields 
\begin{equation}\label{Main-Estloc4}
\|\tilde{N} (\nabla u)\|^2_{L^{2}( \Delta^i)}\le C\nu^2|\Delta^i| +C\alpha^{-1}\|\tilde{N} (\nabla u)\|^2_{L^{2}( 2m\Delta^i)}.
\end{equation}
 Summing over all indices $i$ (using finite overlap of the Whitney cubes $(2m\Delta^i)$) finally yields
 \eqref{Main-Estloc5zz}.

From this as in \cite{DK} one can conclude that there exists $\delta_0>0$ such that for all $0<\delta<\delta_0$ there is $C(\delta)>0$ such that
\begin{equation}\label{eqLP2}
\int_{{\mathbb R}^{n-1}}\left[\tilde{N} (\nabla u)(x')\right]^{2+\delta}\,dx'\le C\int_{{\mathbb R}^{n-1}}|\nabla_T f(x')|^{2+\delta}\,dx'.
\end{equation}
From this $L^{2+\delta}$ solvability of the Regularity problem in Theorem \ref{S3:T3} follows. We claim this extrapolation result also improves Lemma \ref{S7:L1}. That is we have
\begin{lemma}\label{S7:L2} Let $\mathcal L$ be as in Theorem \ref{S3:T3} on the domain ${\mathbb R}^n_{+}$. There exists large $a>0$ such that for  $2\le p<2+\delta$ we have the following.
If $\mathcal O$ is a Lipschitz domain defined by \eqref{Odom} and $u$ is any energy solution of $\mathcal L u=0$ with the Dirichlet boundary datum $\nabla_T f\in L^{p}(\partial \mathcal O;{\BBR}^N)$ then the following estimate holds:
\begin{equation}\label{Main-Estlocxx2q}
\|\tilde{N}_{a/2} (\nabla u)\|_{L^{p}(\Delta_d)}\leq C\|\nabla_T f\|_{L^{p}(\partial{\mathcal O\cap \overline{T(\Delta_{md})}};{\BBR}^N)}+Cd^{(n-1)/p}\sup_{x\in \mathcal O\cap\{\delta(x)>d\}}W(x),
\end{equation}
where $\delta(x)=\mbox{dist}(x,\partial{\mathbb R}^n_+)$, $W(x)=\left(\dint_{B_{\delta(x)/4}(x)}|\nabla u(y)|^2 dy)\right)^{1/2}$ and $m=m(a)>1$ is sufficiently large.
\end{lemma}

This can be seen as follows. For 
$$\nu>\nu_0=2 \sup_{x\in \mathcal O\cap\{\delta(x)>d\}}W(x)$$
the sets $E_\nu$ defined above consist of union of disjoint open subsets. Those subsets that intersect 
the set $\Delta_d$ must be localized to some enlargement of the set $\Delta_d$, say $\Delta_{md}$ which 
leads to a modified estimate  \eqref{Main-Estloc5zz}, namely
\begin{equation}\label{Main-Estloc5-y}
\int_{\tilde{E}_\nu\cap\{g\chi_{\Delta_{md}}\le\nu\}}\left[\tilde{N} (\nabla u)(x')\right]^2\,dx'\le C\nu^2|\tilde{E}_\nu| +C\alpha^{-1}\int_{\tilde{E}_\nu}\left[\tilde{N} (\nabla )(x')\right]^2\,dx',
\end{equation}
where $\tilde{E}_\nu$ is the union of disjoint subsets of $E_\nu$ that intersect $\Delta_d$. Because $\tilde{E}_\nu$ is localized to $\Delta_{md}$ we can truncate $g$ to $\Delta_{md}$ as well in \eqref{Main-Estloc5-y}.

For $\nu\le\nu_0$ we use a trivial estimate
$$|E_\nu\cap \Delta_d|\le |\Delta_d|\approx d^n.$$
The sets $E_\nu$ for $\nu\le \nu_0$ contribute to the value of 
$$\int_{\Delta_d}\left[\tilde{N} (\nabla u)(x')\right]^{p}\,dx'$$
at most $d^n \nu_0^{p}$ which is fine as it is exactly the second term on the righthand side of \eqref{Main-Estlocxx2q} after raising \eqref{Main-Estlocxx2q} to the $p$-th power. 

For $\nu>\nu_0$ the real variable argument of \cite{DK} is used. From this Lemma \ref{S7:L2} does holds.

Finally, extrapolation in the interval $p\in (2-\varepsilon,2)$ and the implied solvability of the Regularity problem in this interval is easier and again is identical to the argument of section 6 of \cite{DHM}. This concludes the proof of Theorem \ref{S3:T3}.

\section{Proof of Theorem \ref{S3:T1}}

The proof is based on ideas of Shen \cite{S2} generalized to our general variable-coefficients settings. We only need to focus of the $L^p$ solvability in the interval $(2,\frac{2(n-1)}{n-3}+\varepsilon)$ and the case $p\in (2-\varepsilon,2]$ already follows from Theorem \ref{S3:T0}.

We used the following abstract result \cite{S3}, see also \cite[Theorem 3.1]{WZ} for a version on an arbitrary bounded domain.

\begin{theorem}\label{th-sh} Let $T$ be a bounded sublinear operator on $L^2({\mathbb R}^{n-1};{\mathbb R}^m)$. Suppose that
for some $p>2$, $T$ satisfies the following $L^p$ localization property. For any ball $\Delta=\Delta_d\subset{\mathbb R}^{n-1}$ and $C^\infty$ function $f$ with supp$(f)\subset{\mathbb R}^{n-1}\subset 3\Delta$ the following estimate holds:
\begin{align}
&\left(|\Delta|^{-1}\int_\Delta|Tf|^p\,dx'\right)^{1/p}\le\label{eq-pf8}\\
&\qquad C\left\{\left(|2\Delta|^{-1}\int_{2\Delta}|Tf|^2\,dx'\right)^{1/2}+\sup_{\Delta'\supset \Delta}\left(|\Delta'|^{-1}\int_{\Delta'}|f|^2\,dx'\right)^{1/2}  \right\},\nonumber
\end{align}
for some $C>0$ independent of $f$. Then $T$ is bounded $L^q({\mathbb R}^{n-1};{\mathbb R}^m)$ for any $2\le q<p$.
\end{theorem}
In our case the role of $T$ is played by the sublinear operator $f\mapsto \tilde{N}_{2,a}(u)$, where $u$ is the solution of the Dirichlet problem ${\mathcal L}u=0$ with boundary data $f$. Clearly, in the Theorem above
the factors $2\Delta$, $3\Delta$ do not play significant role. Hence if we establish estimate \eqref{eq-pf8}
with $2\Delta$ replaced by $m\Delta$ with $f$ vanishing on $(m+1)\Delta$ for some $m>1$ the claim of the Theorem will remain to hold.

It suffices again to work on $\Omega={\mathbb R}^n_+$.
Clearly, our operator $T:f\mapsto \tilde{N}_{2,a}(u)$ is sublinear and bounded on $L^2$
by Theorem \ref{S3:T0}, for coefficients with small Carleson norm $\mu$. To prove \eqref{eq-pf8} we shall establish the following reverse H\"older inequality, following the idea of Shen \cite{S2}.
\begin{equation}
\left(\frac1{|\Delta|}\int_\Delta|\tilde{N}_{2,a}(u)|^p\,dx'\right)^{1/p}\le\label{eq-pf9}
C\left(\frac1{|5m\Delta|}\int_{5m\Delta}|\tilde{N}_{2,a}(u)|^2\,dx'\right)^{1/2},
\end{equation}
where $f=u\big|_{\partial{\mathbb R}^n_+}$ vanishes on $5m\Delta$. Here $m$ is determined by Lemma \ref{S7:L2}.
Having this by Theorem \ref{th-sh} we have for any $q\in [2,p)$ the estimate
\begin{equation}\label{eq-pf10}
\|\tilde{N}_{2,a}(u)\|_{L^{q}({\BBR}^{n-1})}\le C\|f\|_{L^{q}({\BBR}^{n-1})},
\end{equation}
which implies $L^q$ solvability of the Dirichlet problem for the operator $\mathcal L$.

It remains to establish \eqref{eq-pf9}. Let us define
\begin{align}\label{eq-pf11}
&{\mathcal M}_1(u)(x')=\sup_{y\in\Gamma_a(x')}\{w(y):\,\delta(y)\le cd\},\\
&{\mathcal M}_2(u)(x')=\sup_{y\in\Gamma_a(x')}\{w(y):\,\delta(y)> cd\}.\nonumber
\end{align}
where $c=c(a)>0$ is chosen such that for all $x'\in\Delta$ if $y=(y_0,y')\in\Gamma_a(x')$ and $y_0=\delta(y)\le cd$ then $y'\in 2\Delta$. Here $d=\mbox{diam}(\Delta)$ and $w$ is the $L^2$ average of $u$
$$w(y)=\left(\dint_{B_{\delta(y)/2}(y)}|u(z)|^2\,dz\right)^{1/2}.$$
It follows that
$$\tilde{N}_{2,a}(u)=\max\{{\mathcal M}_1(u),{\mathcal M}_2(u)\}.$$
We first estimate ${\mathcal M}_2(u)$. Pick any $x'\in\Delta$.
For any $y\in\Gamma(x')$ with $\delta(y)>cd$ it follows that
for a large subset $A$ of $2\Delta$ (of size comparable to $2\Delta$) we have
$$z'\in A\quad\Longrightarrow\quad y\in\Gamma_a(z')\quad\Longrightarrow\quad w_2(y)\le \tilde{N}_{2,a}(u)(z').$$
Hence for any $x'\in\Delta$ 
$${\mathcal M}_2(u)(x')\le C\left(\frac1{|2\Delta|}\int_{2\Delta} \left[\tilde{N}_{2,a}(u)(z')\right]^2\,dz'\right)^{1/2}.$$
It remains to estimate ${\mathcal M}_1(u)$ on $\Delta$.

We write
$$u(x_0,x')-u(0,y')=\int_{0}^{1}\frac{\partial u}{\partial s}(sx_0,(1-s)y'+sx')\,ds.$$
Let $K=\{(y_0,y'):y'\in \Delta\mbox{ and }cd<y_0<2cd\}$. Using the previous line and the fact that $u$ vanishes on $3\Delta$ we have for any $x'\in\Delta$
\begin{equation}\label{eq-pf12}
{\mathcal M}_1(u)(x')\le \sup_{K}w\,+\,C\int_{2\Delta}\frac{\tilde{N}_{2,a/2}(\nabla u)(y')}{|x'-y'|^{n-2}}dy'.
\end{equation}
By the fractional integral estimate, this implies that
\begin{equation}\label{eq-pf13}
\left(\frac1{|\Delta|}\int_\Delta[{\mathcal M}_1(u)(x')]^p\,dx'\right)^{1/p}\le \sup_{K}w\,+\,Cd
\left(\frac1{|2\Delta|}\int_{2\Delta}[\tilde{N}_{2,a/2}(\nabla u)(x')]^q\,dx'\right)^{1/q},
\end{equation}
where $\frac1p=\frac1q-\frac1{n-1}$ and $1<q<n-1$.\vglue1mm

To further estimate \eqref{eq-pf13} we need to use the local solvability of the $L^q$ Regularity problem from Lemma \ref{S7:L2}. We apply the lemma to the domain 
${\mathcal O}_\tau={\mathcal O}_{\tau\Delta,a}$, where $\Delta$ is as in \eqref{eq-pf13} and $\tau\in [2,3]$. This gives us
\begin{equation}\label{Main-Estlocxx2}
\|\tilde{N}_{a/2} (\nabla u)\|_{L^{q}(2\Delta)}\leq C\|\nabla_T f\|_{L^{q}(\partial{\mathcal O\cap \overline{T(\tau m\Delta)}})}+Cd^{(n-1)/q}\sup_{x\in {\mathcal O}_\tau \cap\{\delta(x)>d\}}W(x).
\end{equation}
Observe first that by the Ciacciopoli's inequality we have for all $x\in{\mathbb R}^n_+$ with $\delta(x)>d$
$$W(x)\le Cd^{-1}w(x).$$
We have intentionally shrunk the size of the ball in the definition of $W$ so that this pointwise estimate holds. Since the $x$ we consider in the supremum in \eqref{Main-Estlocxx2} is in $\mathcal O_\tau$ it then follows
\begin{equation}\label{eq-pf16}
\sup_{x\in {\mathcal O}_\tau\cap\{\delta(x)>d\}}W(x)\lesssim d^{-1}\left(\frac1{|3\Delta|}\int_{3\Delta} \left[\tilde{N}_{2,a}(u)(z')\right]^2\,dz'\right)^{1/2}.
\end{equation}
We use this in \eqref{Main-Estlocxx2}, integrate \eqref{Main-Estlocxx2} in $\tau$ over the interval $[2,3]$ and divide by $d^{(n-1)/q}$. This gives after using the fact that $u=0$ on $5m\Delta$:
\begin{align}\label{eq-pf17}
&\left(\frac1{|2\Delta|}\int_{2\Delta}[\tilde{N}_{2,a/2}(\nabla u)(x')]^q\,dx'\right)^{1/q}\\
\lesssim &\quad
\left(\frac1{T(3m\Delta)}\iint_{T(3m\Delta)}|\nabla u|^q\,dx\right)^{1/q} +d^{-1}\left(\frac1{|3\Delta|}\int_{3\Delta} \left[\tilde{N}_{2,a}(u)(z')\right]^2\,dz'\right)^{1/2}.\nonumber
\end{align}
We have also used the trivial estimate $|\nabla_Tu|\le|\nabla u|$ on $\partial {\mathcal O}_\tau\cap{T(3m\Delta)}$. 

The Caccioppoli's inequality holds for all strongly elliptic systems. Its well known consequence is the higher integrability for $\nabla u$ (c.f. \cite{Gi}) which implies that for some $q>2$ we have
$$\left(\frac1{T(3m\Delta)}\iint_{T(3m\Delta)}|\nabla u|^q\,dx\right)^{1/q}\lesssim \left(\frac1{T(4m\Delta)}\iint_{T(4m\Delta)}|\nabla u|^2\,dx\right)^{1/2}.$$
It follows by the boundary Ciacciopoli's inequality
\begin{align}\nonumber
\left(\frac1{T(4m\Delta)}\iint_{T(4m\Delta)}|\nabla u|^2\,dx\right)^{1/2}&\lesssim d^{-1}\left(\frac1{T(5m\Delta)}\iint_{T(5m\Delta)}|u|^2\,dx\right)^{1/2}\\\nonumber
&\lesssim d^{-1}\left(\frac1{|5m\Delta|}\int_{5m\Delta} \left[\tilde{N}_{2,a}(u)(z')\right]^2\,dz'\right)^{1/2}.\nonumber
\end{align}
This combined with \eqref{eq-pf17} gives us
$$\left(\frac1{|2\Delta|}\int_{2\Delta}[\tilde{N}_{2,a/2}(\nabla u)(x')]^q\,dx'\right)^{1/q}\lesssim d^{-1}\left(\frac1{|5m\Delta|}\int_{5m\Delta} \left[\tilde{N}_{2,a}(u)(z')\right]^2\,dz'\right)^{1/2}.
$$
Finally, inserting this estimate into \eqref{eq-pf13} yields
\begin{equation}\label{eq-pf19}
\left(\frac1{|\Delta|}\int_\Delta[{\mathcal M}_1(u)(x')]^p\,dx'\right)^{1/p}\le C\left(\frac1{|5m\Delta|}\int_{5m\Delta} \left[\tilde{N}_{2,a}(u)(z')\right]^2\,dz'\right)^{1/2},
\end{equation}
where $\frac1p=\frac1q-\frac1{n-1}$ and $1<q<n-1$. Here $q>2$ is such that improved integrability of $\nabla u$ holds and also $q<2+\delta$ with $\delta$ as in Lemma \ref{S7:L2}. This implies in dimensions $2$ and $3$ that we can consider any $2<p<\infty$, while in dimensions $n\ge 4$ we can have
$2<p\le p_{\max}=q(n-1)/(n-1-q)$. Observe that always $p_{max}>2(n-1)/(n-3)$. From this claim of Theorem \ref{S3:T1} follows as we have established \eqref{eq-pf9} for such values of $p$.\qed

\section{Proof of Corollary \ref{C:lame}}

We discuss here how the Lam\'e system \eqref{ES-lame2} fits into the framework we have introduced in this paper, in particular when the weaker oscillation condition \eqref{Car_lame2} is imposed.
The case when stronger condition \eqref{Car_lame} is assumed is discussed extensively in section 7 of \cite{DHM}.

Suppose that $\lambda,\mu$ are as in Corollary \ref{C:lame} and let \eqref{Car_lame2} be a small Carleson measure. Let $\mathcal L_1$ be as in section \ref{S4} defined by \eqref{ES-new2} where
\begin{equation}\label{eqLC}
\bar{A}_{ij}^{\alpha\beta}(x)=\mu(x)\delta_{ij}\delta_{\alpha\beta}+\lambda(x)\delta_{i\alpha}\delta_{j\beta}+\mu(x)\delta_{i\beta}\delta_{j\alpha},\quad \bar{B}^{\alpha\beta}_i(x)=0.
\end{equation}

Let $\mathcal L$ be as in section \ref{S4} defined by \eqref{ES-new} with coefficients
\begin{equation}\label{eqLCs}
{A}_{ij}^{\alpha\beta}(x)=\tilde\mu(x)\delta_{ij}\delta_{\alpha\beta}+\tilde\lambda(x)\delta_{i\alpha}\delta_{j\beta}+\tilde\mu(x)\delta_{i\beta}\delta_{j\alpha},\quad {B}^{\alpha\beta}_i(x)=0,
\end{equation}
where
\begin{equation}\label{eq-lmu}
\tilde{\lambda}(x)= \int_{{\mathbb R}^n} \lambda(y)\varphi_{\rho(x)}(x-y)dy,\quad \tilde{\mu}(x)= \int_{{\mathbb R}^n} \mu(y)\varphi_{\rho(x)}(x-y)dy,
\end{equation}
Here (as in section \ref{S6}) $\varphi$ is a smooth real, nonnegative bump function on ${\mathbb R}^n$ supported in the ball  $B_{1/2}(0)$ such that $\int \varphi=1$ and
$\varphi_t(y) = t^{-n}\varphi(y/t)$. By $\rho(x)$ we denote a mollified distance function (i.e. $\rho(x)\approx\delta(x)$ but $\rho$ is smooth in the interior of $\Omega$). By the same argument as in section
\ref{S6} the coefficients of $\mathcal L$ satisfy the small Carleson condition \eqref{Car_lame}.

Observe also that if \eqref{Cond-lame} holds for the pair $(\lambda,\mu)$ then it also does for $(\tilde\lambda,\tilde\mu)$.

The operators $\mathcal L,\,\mathcal L_1$ do not quite fit the framework of sections \ref{S4}-\ref{SS:43} yet.  We have as required \eqref{Car-m}  a small Carleson measure but $A_{0j}^{\alpha\beta}=\delta_{\alpha\beta}\delta_{0j}$ does not hold and hence there results of sections \ref{S4}-\ref{SS:43} do not apply to $\mathcal L,\,\mathcal L_1$ without further adjustments.

Recall the two transformations (2.19)-(2.20) and (2.24) of \cite{DHM} that modify our operator there to achieve $A_{0j}^{\alpha\beta}=\delta_{\alpha\beta}\delta_{0j}$. We outline it briefly here.

From now on let $M(x)=(A^{\alpha\beta}_{00}(x))_{\alpha,\beta=1}^n$ be a minor $n\times n$ matrix (which is invertible by \eqref{Cond-lame}. Following (2.19)-(2.20) of \cite{DHM} let

\begin{equation}\nonumber
\widehat{A}^{\alpha\beta}_{ij}:=\sum_{\gamma=0}^{n-1}\left[M^{-1}\right]^{\alpha\gamma}A_{ij}^{\gamma\beta},\quad \widehat{B}^{\alpha\beta}_{i}:=\sum_{\gamma=0}^{n-1}\left(\left[M^{-1}\right]^{\alpha \gamma} 
B_{i}^{\gamma\beta}-\sum_{k=0}^{n-1}\partial_k\left([M^{-1}]^{\alpha\gamma}\right)A_{ki}^{\gamma\beta}\right),
\end{equation}
\begin{equation}\nonumber
\widehat{\bar A}^{\alpha\beta}_{ij}:=\sum_{\gamma=0}^{n-1}\left[M^{-1}\right]^{\alpha\gamma}\bar{A}_{ij}^{\gamma\beta},\quad \widehat{\bar B}^{\alpha\beta}_{i}:=\sum_{\gamma=0}^{n-1}\left(\left[M^{-1}\right]^{\alpha \gamma} 
\bar{B}_{i}^{\gamma\beta}-\sum_{k=0}^{n-1}\partial_k\left([M^{-1}]^{\alpha\gamma}\right)\bar{A}_{ki}^{\gamma\beta}\right),
\end{equation}
and by (2.24) of \cite{DHM}  set
\begin{eqnarray}\label{yr4DDF}
&&\overline{A_{ij}^{\alpha\beta}}:=
\begin{cases}
\widehat A_{ij}^{\alpha\beta},&\quad\mbox{if $i,j>0$ or $i=j=0$,}
\\[6pt]
\widehat A_{ij}^{\alpha\beta}+\widehat A_{ji}^{\alpha\beta},&\quad\mbox{if $i>0$ and $j=0$,}
\\[6pt]
0,&\quad\mbox{if $i=0$ and $j>0$,}
\end{cases}\\\nonumber
&&\overline{B_{i}^{\alpha\beta}}:=
\begin{cases}
\widehat{B}_{i}^{\alpha\beta}+\sum_{j=0}^{n-1}\partial_j(\widehat A^{\alpha\beta}_{0j})&\quad\mbox{if $i=0$,}\\[6pt]
\widehat{B}_{i}^{\alpha\beta}-\partial_0(\widehat A^{\alpha\beta}_{0i})&\quad\mbox{if $i>0$,}
\end{cases}
\end{eqnarray}
\begin{eqnarray}\label{yr4DDFA}
&&\overline{\bar{A}_{ij}^{\alpha\beta}}:=
\begin{cases}
\widehat {\bar{A}}_{ij}^{\alpha\beta},&\quad\mbox{if $i,j>0$ or $i=j=0$,}
\\[6pt]
\widehat {\bar{A}}_{ij}^{\alpha\beta}+\widehat A_{ji}^{\alpha\beta},&\quad\mbox{if $i>0$ and $j=0$,}
\\[6pt]
0,&\quad\mbox{if $i=0$ and $j>0$.}
\end{cases}
\\\nonumber
&&\overline{\bar{B}_{i}^{\alpha\beta}}:=
\begin{cases}
\widehat{\bar B}_{i}^{\alpha\beta}+\sum_{j=0}^{n-1}\partial_j(\widehat A^{\alpha\beta}_{0j})&\quad\mbox{if $i=0$,}\\[6pt]
\widehat{\bar B}_{i}^{\alpha\beta}-\partial_0(\widehat A^{\alpha\beta}_{0i})&\quad\mbox{if $i>0$.}
\end{cases}
\end{eqnarray}
The setup here is rather delicate, notice that we define the new coefficients so that a derivative only falls on the mollified coefficients (never on the coefficients defined using the original $\lambda$, $\mu$ as these do not have sufficient smoothness).\vglue1mm

Finally, let $\overline{\mathcal L_1}$ be the operator associated with coefficients $\overline{\bar{A}_{ij}^{\alpha\beta}},\,\overline{\bar{B}_{i}^{\alpha\beta}}$ and let $\overline{\mathcal L}$ be the operator 
associated with coefficients $\overline{{A}_{ij}^{\alpha\beta}},\,\overline{{B}_{i}^{\alpha\beta}}$.

We have that $\mathcal Lu=0\,\Leftrightarrow \overline{\mathcal L}u=0$ and 
$\mathcal L_1u=0\,\Leftrightarrow \overline{\mathcal L_1}u=0$. Furthermore by section 2 of \cite{DHM}
we also have that the coefficients of $\overline{\mathcal L}$ satisfy the condition $\overline{A_{0j}^{\alpha\beta}}=\delta_{\alpha\beta}\delta_{0j}$. It follows that the framework of sections \ref{S4}-\ref{SS:43} applies to these two modified operators and in particular the conclusions of Lemma \ref{S3:L4-alt1}, Corollary \ref{S4:C1-alt1} and Lemma \ref{S3:L8-alt1} hold, provided the operator $\overline{\mathcal L}$ is strongly elliptic. That question has however been dealt with in section 7 of \cite{DHM} (Lemma 7.1) and in particular it has been shown that \eqref{Cond-lame} is a sufficient condition for the strongly ellipticity (although some further changes of the coefficients similar to \eqref{yr4DDF}-\eqref{yr4DDFA} might be needed).

From this claims in the first part of Corollary \ref{C:lame} up to the estimate \eqref{Main-Est-LM} follow.
The remaining claims follow directly from Theorems \ref{S3:T1}  and \ref{S3:T3} (we no longer have to work with two operators $\mathcal L,\,\mathcal L_1$)  as we can skip the modification procedure \eqref{eq-lmu} since our Lam\'e coefficients already have the necessary smoothness. 

\vglue3mm

\end{document}